\documentclass[british]{amsart}
\usepackage[T1]{fontenc}
\usepackage[utf8]{inputenc}
\usepackage[dvipsnames]{xcolor}
\usepackage[numbers]{natbib}
\usepackage{amstext,amsthm,amssymb}
\usepackage{xparse,mathrsfs,mathtools}
\usepackage{csquotes}
\usepackage{bookmark}
\usepackage{booktabs}
\usepackage[capitalise,nameinlink]{cleveref}
\usepackage{babel}
\usepackage[yyyymmdd,hhmmss]{datetime}
\usepackage{hyperref}
\usepackage{lmodern,microtype}

\hypersetup{
	unicode=true,%
	pdfdisplaydoctitle=true,%
	pdfpagemode=UseOutlines,%
	breaklinks=true,%
	pdfencoding=unicode,psdextra,
	pdfcreator={},
	pdfborder={0 0 0},
	hidelinks
}

\newtheorem{Theorem}{Theorem}[section]
\newtheorem{Proposition}[Theorem]{Proposition}
\newtheorem{Lemma}[Theorem]{Lemma}
\newtheorem{Corollary}[Theorem]{Corollary}
\newtheorem{Definition}[Theorem]{Definition}

\DeclarePairedDelimiter\parentheses{\lparen}{\rparen}

\DeclarePairedDelimiter\braces{\lbrace}{\rbrace}
\DeclarePairedDelimiter\abs{\lvert}{\rvert}
\DeclarePairedDelimiter\norm{\lVert}{\rVert}

\newcommand{\Norm}{\mathrm{N}}

\newcommand{\QQ}{\mathbb{Q}}
\newcommand{\RR}{\mathbb{R}}

\newcommand{\cupdot}{\mathbin{\mathaccent\cdot\cup}}

\DeclareMathOperator{\eOpname}{e}
\NewDocumentCommand\e{ s O{} m }{
	\IfBooleanTF{#1}{%
		\eOpname_{#2}\parentheses[\big]{#3}%
	}{\eOpname_{#2}\parentheses{#3}}%
}

\numberwithin{equation}{section}
\crefformat{equation}{#2(#1)#3}
\crefformat{enumi}{#2(#1)#3}

\usepackage[dvipsnames]{xcolor}
\usepackage{tikz}

\usetikzlibrary{calc}
\usetikzlibrary{arrows}
\usetikzlibrary{backgrounds}
\usetikzlibrary{positioning}
\usetikzlibrary{decorations.pathreplacing}

\definecolor{othercolor}{gray}{0.80}

\definecolor{othercolorTwo}{rgb}{1,1,1}

\title{On the distribution of 
$\alpha p$ modulo one in quadratic number fields}


\author{Stephan~Baier}
\address{Stephan~Baier\\%
	Ramakrishna Mission Vivekananda Educational Research Institute\\%
	Department of Mathematics\\%
	G.\ T.\ Road, PO~Belur Math, Howrah, West Bengal~711202\\%
	India}
\email{stephanbaier2017@gmail.com}
\urladdr{https://www.researchgate.net/profile/Stephan\_Baier2}

\author{Dwaipayan~Mazumder}
\address{Dwaipayan~Mazumder\\%
	Ramakrishna Mission Vivekananda Educational Research Institute\\%
	Department of Mathematics\\%
	G.\ T.\ Road, PO~Belur Math, Howrah, West Bengal~711202\\%
	India}
\email{dwmaz.1993@gmail.com}

\author{Marc~Technau}
\address{Marc~Technau\\%
	Graz University of Technology\\%
	Institute of Analysis and Number Theory\\%
	Kopernikusgasse~24/II\\%
	8010~Graz\\%
	Austria}
\email{mtechnau@math.tugraz.at}
\urladdr{https://www.math.tugraz.at/\string~mtechnau/}

\begin{document}

\begin{abstract}
	We investigate the distribution of $\alpha p$ modulo one in quadratic number fields $\mathbb{K}$ with class number one, where $p$ is restricted to prime elements in the ring of integers of $\mathbb{K}$. Here we improve the relevant exponent $1/4$ obtained by the first and third named authors for imaginary quadratic number fields~\cite{BT} and by the first and second named authors for real quadratic number fields~\cite{BM} to $7/22$. This generalizes a result of Harman ~\cite{HarZi} who obtained the same exponent $7/22$ for $\mathbb{Q}(i)$ by extending his method which gave this exponent for $\mathbb{Q}$~\cite{harman1996on-the-distribu}.  Our proof is based on an extension of his sieve method to arbitrary number fields. Moreover, we need an asymptotic evaluation of certain smooth sums over prime ideals appearing in~\cite{BM}, for which we use analytic properties of Hecke L-functions with Gr\"o\ss encharacters. 
\end{abstract}

\maketitle


\tableofcontents

\section{Introduction}
The distribution of $\alpha p$ modulo one, where $p$ runs over the rational primes and $\alpha$ is a fixed irrational real, has received a lot of attention. Throughout this article, for $x$ a real number, we denote by $||x||$ the distance of $x$ to the nearest integer. It is natural to ask for which exponents $\nu>0$ one can establish the infinitude of primes $p$ satisfying $||\alpha p||\le p^{-\nu}$. The latest record in this regard is Kaisa Matom\"aki's landmark result $\nu=1/3-\varepsilon$ (see~\cite{Mat}) which presents the limit of currently known technology. An earlier result in this direction was obtained by Harman (see~\cite{harman1996on-the-distribu}) who proved that $\nu=7/22$ is an admissible exponent using his sieve method which has become a standard tool in analytic number theory. Recently, Harman extended his method to Gaussian integers~\cite{HarZi}. The first and third-named authors investigated this problem more generally in the context of imaginary quadratic fields of class number one~\cite{BT}, establishing an exponent which corresponds to $\nu=1/4$. An analog of this result was proved by the first and second-named authors for real quadratic fields in~\cite{BM}. In this article, we improve the exponent 1/4 in the two last-mentioned results to Harman's exponent 7/22. In fact, this exponent appears halfed in the relevant results below, Theorem~\ref{imag}, Corollary~\ref{imagcor} and Theorem~\ref{realcase}, because here we deal with two-dimensional approximation problems. To this end, and with other applications in mind, we shall prove weighted versions of Harman's asymptotic and 
lower bound sieves for general number fields, where integers are replaced by ideals of the relevant rings of algebraic integers. The lower bound sieve will give rise to our improvement. In fact, versions of Harman's sieve for quadratic number fields would suffice for our application, but it is not much of an extra effort to make them work for general number fields. 

\section{New results on restricted Diophantine approximation for quadratic number fields}
Before we work out said versions of Harman's sieve, we formulate our aforementioned improvements on Diophantine approximation with primes in the context of quadratic number fields. These results will be proved in section~\ref{quadproofs} by applying our number field version of Harman's lower bound sieve. Throughout this article, we use the following notations.

\begin{itemize}
\item By $\mathbb{K}$, we denote a number field.
\item By $\mathcal{O}$, we denote the ring of algebraic integers in $\mathbb{K}$. 
\item By $\mathcal{N}(a)$, we denote the norm over $\mathbb{Q}$ of $a\in \mathcal{O}$ and by $\mathcal{N}(\mathfrak{a})$ the norm of an ideal $\mathfrak{a}$ in $\mathcal{O}$.
\item We use the Landau notation $O(...)$ and Vinogradov's notation $\ll$ in the usual way. All implied constants may depend on the number field $\mathbb{K}$.
\end{itemize}
 
\subsection{Imaginary quadratic number fields}
In this subsection, assume that $\mathbb{K}$ is an imaginary quadratic number field. Then  
$\mathcal{O}$ is a free $\mathbb{Z}$-module of rank $2$ and there is some $\omega\in\mathcal{O}$ such that $\{1,\omega\}$ is a $\mathbb{Z}$-basis of $\mathcal{O}$. 
Since, by assumption, $\mathbb{K}\not\subseteq\mathbb{R}$, it follows that $\Im\omega\neq 0$.
In particular, $\{1,\omega\}$ turns out to be an $\mathbb{R}$-basis of $\mathbb{C}$ and, given some $\varrho\in\mathbb{C}$, we write 
$\Re_{\omega}(\varrho)$ and $\Im_{\omega}(\varrho)$ for the unique real numbers satisfying
\[
	\varrho = \Re_{\omega}(\varrho) + \Im_{\omega}(\varrho)\omega.
\]
With this notation, we put
\[
	||\varrho||_\omega
	= \max\{ ||\Re_{\omega}(\varrho)||, ||\Im_{\omega}(\varrho)|| \}.
\]
We shall prove the following.

\begin{Theorem} \label{imag}
Let $\mathbb{K}\subset\mathbb{C}$ be an imaginary quadratic number field of class number one and let $\mathcal{O}$ be its ring of integers. 
Suppose that $\alpha$ is a complex number such that $\alpha\notin\mathbb{K}$. Then, for any $\varepsilon>0$, there exists an infinite sequence of prime 
elements $p\in\mathcal{O}$ such that
	\begin{equation}\label{mainresult}
		||p\alpha||_\omega \leq \mathcal{N}(p)^{-7/44+\varepsilon}.
	\end{equation}
\end{Theorem} 

This improves~\cite[Theorem 2.1]{BT} in which the above result with the exponent $1/8$ in place of $7/44$ was established. As already pointed out in Remark (1) after \cite[Theorem 2.2.]{BT}, instead of considering the homogeneous condition $\norm{p\alpha}_\omega < \delta$ in the above theorem, one can also consider an inhomogeneous version, namely $\norm{p\alpha+\beta}_\omega < \delta$, where $\beta$ is an arbitrary complex number. Our methods should be capable of handling such shifts $\beta$ but we chose not to implement this for the sake of readability of this article.  

In the next subsection, we formulate a version of this theorem for real-quadratic fields which improves the main result in~\cite{BM}. Again, there are no fundamental constraints in working out an inhomogeneous version, but we don't deal with this in the present article. To see the connection to Theorem~\ref{imag} above clearly, we formulate this theorem in a slightly different form below. 

\begin{Corollary} \label{imagcor}
Under the conditions of Theorem~\ref{imag}, there exist infinitely many non-zero prime ideals $\mathfrak{p}$ in $\mathcal{O}$ such that 
\begin{equation*}
\left|\alpha-\frac{a}{p}\right|\le \mathcal{N}(\mathfrak{p})^{-1/2-7/44+\varepsilon} 
\end{equation*}
for some generator $p$ of $\mathfrak{p}$ and $a\in \mathcal{O}$. 
\end{Corollary}

\begin{proof}
For any complex number $z$, we now denote by $||z||$ the Euclidean distance of $z$ to the nearest algebraic integer $a\in \mathcal{O}$. This extends our definition of $||x||$ for real $x$ from the beginning of this paper. Then
$$
||z||\ll_{\omega} ||z||_{\omega}
$$
for any complex $z$. 
Hence, using Theorem~\ref{imag}, we 
have 
$$
||p\alpha||\le \mathcal{N}(p)^{-7/44+2\varepsilon} 
$$
for infinitely many prime elements $p$ in $\mathcal{O}$. Therefore,
$$
|p\alpha-a|\le \mathcal{N}(p)^{-7/44+2\varepsilon}
$$
for infinitely pairs $(p,a)\in \mathcal{O}\times \mathcal{O}$ with $p$ prime, which is equivalent to 
$$
\left|\alpha-\frac{a}{p}\right|\le \mathcal{N}(p)^{-1/2-7/44+2\varepsilon}.
$$
The result now follows upon redefining $\varepsilon$ and noting that $\mathcal{O}$ has a finite group of units. 
\end{proof}

\subsection{Real quadratic number fields} 
In this subsection, we assume, as in~\cite{BM}, that $\mathbb{K}=\mathbb{Q}(\sqrt{d}) \subset \RR$ is a real quadratic number field of class number one, where $d>0$ is square-free and $d\equiv 3\bmod{4}$. Moreover, we use the following notations. 

\begin{itemize} 
\item We denote the two embeddings of $\mathbb{K}$, given by the identity and conjugation, by 
$$
\sigma_1(\alpha+\beta\sqrt{d})\coloneqq\alpha+\beta\sqrt{d}
$$ 
and 
$$
\sigma_2(\alpha+\beta\sqrt{d})\coloneqq\alpha-\beta\sqrt{d}.
$$
\item We write $\sigma(\mathbb{K})\coloneqq\{(\sigma_1(\gamma),\sigma_2(\gamma)) : \gamma\in \mathbb{K}\}$.
\item $(x_1,x_2)$ is a pair of real numbers which does belong to $\sigma(\mathbb{K})$. 
\end{itemize}

To formulate our result in real quadratic fields, we need to introduce a notion of ``good'' and ``bad'' pairs $(x_1,x_2)$. It was proved in~\cite{BM} that almost all elements of $\mathbb{R}^2\setminus \sigma(\mathbb{K})$ are good. 
By a version of Dirichlet's approximation theorem for $\mathbb{K}$ (see~\cite[Corollary~3]{BM}), there are infinitely many natural numbers $W$ such that 

\begin{equation} \label{theabove}
\left|x_i- \frac{\sigma_i(u+v\sqrt{d})}{\sigma_i(f+g\sqrt{d})}\right|\le \frac{1}{|\mathcal{N}(f+g\sqrt{d})|} \quad \mbox{ for }
i=1,2
\end{equation}
for suitable $u,v,f,g\in \mathbb{Z}$ with $u+v\sqrt{d}$ and $f+g\sqrt{d}$ coprime in $\mathcal{O}$ and $|\mathcal{N}(f+g\sqrt{d})|=W$. 
In the following, we define what we understand as $\eta$-good, good and bad pairs $(x_1,x_2)$.

\begin{Definition} \label{goodbad} For $\eta>0$, we call $(x_1,x_2)$ $\eta$-good if there is an infinite sequence of natural numbers $W$ such that~\eqref{theabove} holds 
with  $u+v\sqrt{d}$ and $f+g\sqrt{d}$ coprime in $\mathcal{O}$,
$$
|{\mathcal{N}(f+g\sqrt{d})}|=W \quad \mbox{ and }  \quad 
{\rm gcd}(f,g)=O\left(W^{\eta}\right),
$$
where {\rm gcd}$(f,g)$ is meant to be the largest natural number dividing both $f$ and $g$. 
We call $(x_1,x_2)$ good if it is $\eta$-good for all $\eta>0$. We call $(x_1,x_2)$ bad if it is not good.
\end{Definition}

To avoid confusions, we note that in~\cite{BM}, we have used the notation $\mathcal{N}(q)$ for the {\it modulus} of the norm of $q$.  
Now we are ready to formulate our main result on the real quadratic case. 

\begin{Theorem} \label{realcase}
Assume that $\mathbb{K}=\mathbb{Q}(\sqrt{d})$ has class number one, where $d$ is a positive square-free integer satisfying $d\equiv 3 \bmod{4}$. Let $\varepsilon$ be any positive real number. Suppose further that $(x_1,x_2)\in \mathbb{R}^2\setminus \sigma(\mathbb{K})$ is $\eta$-good
in the above sense. Set
\begin{equation} \label{nueq}
\nu\coloneqq\frac{7/44-\eta}{1+2\eta}.
\end{equation}
Then there exist infinitely many non-zero prime ideals $\mathfrak{p}$ in the ring $\mathcal{O}$ of integers of $\mathbb{K}$
such that 
\begin{equation*}
\left|x-\frac{\sigma_i(a)}{\sigma_i(p)}\right|\le \mathcal{N}(\mathfrak{p})^{-1/2-\nu+\varepsilon} \quad \mbox{ for } i=1,2
\end{equation*}
for some generator $p$ of $\mathfrak{p}$ and $a\in \mathcal{O}$. 
If $(x_1,x_2)$ is good, then the above holds with $\nu=7/44$. 
\end{Theorem}

This result provides an analog of Corollary~\ref{imagcor} for good $(x_1,x_2)$ in real-quadratic fields and improves~\cite[Theorem 5]{BM}, in which Theorem~\ref{realcase} with 
$$
\nu=\frac{1/8-\eta}{1+2\eta}
$$
 in place of~\eqref{nueq} for $\eta$-good pairs $(x_1,x_2)$ and $\nu=1/8$ in place of $\nu=7/44$ for good pairs $(x_1,x_2)$ was established.\\ \\
{\bf Remark 1:}  As remarked in~\cite{BM}, the condition $d\equiv 3 \bmod{4}$ in Theorem~\ref{realcase} is for convenience of proof but should not be essential. (Note that if $d\equiv 3 \bmod{4}$, then the ring of integers equals $\mathbb{Q}[\sqrt{d}]$, which is particularly convenient. Moreover, $d$ is odd which helps at some place in~\cite{BM}.) It should not be difficult to remove this condition but we have not worked out the details. In~\cite{BM}, we also described how to remove the ``class number 1'' condition in the real quadratic case. Similar arguments apply in the imaginary quadratic case. If the class number is not equal to 1, then one needs to assume in addition that the prime ideals $\mathfrak{p}$ are principal. \\ \\
{\bf Remark 2:} It may be conjectured that all pairs $(x_1,x_2)\in \mathbb{R}^2\setminus \sigma(\mathbb{K})$ are good. Indeed, the first and second named authors have observed that this would follow from a version of Chowla's conjecture on the least prime in an arithmetic progression for $\mathbb{K}$. We will address this in upcoming work. \\ \\
{\bf Remark 3:} We would like to point out a small oversight in \cite[subsection 9.2.]{BM} which can be easily fixed and doesn't affect the results in \cite{BM}. After equation (82) in the said article, we deduced that $t|2\sqrt{d}$. However, in what follows, there are more cases to be considered than just $t\approx 1$ and $t\approx \sqrt{d}$. In fact, all possible cases of divisors $t$ of $2\sqrt{d}$ in $\mathcal{O}_K$ such that $t$ is not divisible by 2 can occur. Recall that we assumed the class number to be 1. If $t$ is of the said form, then we have
$$
t\approx P_1\cdots P_g,
$$ 
where $P_1,...,P_g$ are distinct prime elements of $O_K$ with $P_i^2$ associated to a rational prime $p_i$ dividing $2d$ (note that all primes dividing $2d$ ramify since the discriminant of $K$ equals $4d$ if $d\equiv 3\bmod{4}$). 
Now let $d':=p_1\cdots p_g$, $d_1:=$gcd$(d,d')$ and $d_2=d/d_1$. In a similar way as in \cite[subsection 9.2. after (84)]{BM}, we deduce that gcd$(2a',W')=d'$. Now, in addition to the case (54) in \cite[subsection 7.4.]{BM} (which corresponds to $t\approx 1$), we get the more general case 
$$
a=Za',\quad b=Zb', \quad W=Z^2W', \quad \mbox{gcd}(2a',W')=d'
$$    
in place of (55). Similarly as in \cite[subsection 7.4.]{BM}, we now write $a''=a'/d_1$ and $W''=W'/d'$ and reduce the congruence (56) to 
$$
(a'')^2d_1\equiv (b')^2d_2 \bmod{ W''}.
$$
(Here we divide the said congruence by $d_1$ and the resulting modulus $W'/d_1$ further by $d'/d_1=1$ or $2$). Multiplying the above with $d_2$, we get
$$
(a'')^2d\equiv (b'd_2)^2 \bmod{ W''}.
$$
Moreover, from (58), we deduce that
$$
b'd_2B\equiv a''Ad \bmod{W''}.
$$
Noting gcd$(a'',W'')=1$, we transform this into
$$
Ad \equiv b'd_2\overline{a''}B\bmod{W''}.
$$
Now the treatment of this case finishes exactly as in \cite[subsection 7.4.]{BM} with $b'd_2$ in place of $b'$, i.e. we make the replacements
\begin{equation*}
\begin{split}
W' & \rightarrow W''\\
a' & \rightarrow b'd_2\\
b' & \rightarrow a''\\
A & \rightarrow B\\
B & \rightarrow Ad.
\end{split}
\end{equation*}
{\bf Acknowledgements.} The authors would like to thank the anonymous referee for his valuable comments. The two first-named authors would like to thank the Ramakrishna Mission Vivekananda Educational and Research Institute for an excellent work environment. 

\section{Harman's asymptotic sieve for number fields}
\subsection{Notations}
In this and the next section, we take $\mathbb{K}$ to be a general number field. No extra assumptions on $\mathbb{K}$ are required to formulate and prove our sieve results. We shall use the following notations.
\begin{itemize}
\item $\mathcal{I}$ denotes the set of non-zero ideals of $\mathcal{O}$.
\item $\mathbb{P}$ denotes the set of all non-zero prime ideals of $\mathcal{O}$.
\item We fix some total order $\prec$ on $\mathbb{P}$ such that $\mathfrak{p}_2 \prec \mathfrak{p}_1$ whenever $\mathcal{N}(\mathfrak{p}_2)< \mathcal{N}(\mathfrak{p}_1)$. 
Many such orderings exist, but we are not concerned about the precise order.
\item If $z>0$, we denote by $\mathfrak{Q}(z)$ the smallest prime ideal $\mathfrak{q}$ with respect to the ordering $\prec$ such that $\mathcal{N}(\mathfrak{q})\ge z$.
\item For $z>0$ and $\mathfrak{q}\in \mathbb{P}$, we set 
\begin{gather*}
	\mathbb{P}(z) \coloneqq \{\mathfrak{p}\in\mathbb{P}: \mathcal{N(\mathfrak{p})}<z\}, \quad \mathbb{P}(\mathfrak{q})\coloneqq \{\mathfrak{p}\in\mathbb{P}: \mathfrak{p}\prec \mathfrak{q}\} \\
	\Pi(z) \coloneqq \prod_{\mathfrak{p}\in\mathbb{P}(z)} \mathfrak{p} \quad 
\mbox{and} \quad \Pi(\mathfrak{q}) \coloneqq \prod_{\mathfrak{p}\in\mathbb{P}(\mathfrak{q})} \mathfrak{p}.
\end{gather*}
By convention, $\Pi(z)=1$ if $\mathbb{P}(z)$ is empty and $\Pi(\mathfrak{q})=1$ if $\mathbb{P}(\mathfrak{q})$ is empty (i.e., $\mathfrak{q}$ is the smallest prime ideal with respect to the ordering $\prec$). 
\item For any two $\mathfrak{a},\mathfrak{b}\in \mathcal{I}$, we write $(\mathfrak{a},\mathfrak{b})=1$ if $\mathfrak{a}$ and $\mathfrak{b}$ are coprime.
\item For any function $w : \mathcal{I} \rightarrow \mathbb{C}$, $\mathfrak{r}\in \mathcal{I}$, $z>0$ and $\mathfrak{q}\in \mathbb{P}$, we set
\begin{equation} \label{Srdef}
S_{\mathfrak{r}}(w,z)\coloneqq\sum_{\substack{\mathfrak{a}\in \mathcal{I}\\(\mathfrak{a},\Pi(z))=1}} w(\mathfrak{ar}) \quad \mbox{and} \quad 
\Phi_{\mathfrak{r}}(w,\mathfrak{q})\coloneqq\sum_{\substack{\mathfrak{a}\in \mathcal{I}\\(\mathfrak{a},\Pi(\mathfrak{q}))=1}} w(\mathfrak{ar}).
\end{equation}
\item $d_k(\mathfrak{a})$ denotes the number of ways to write the ideal $\mathfrak{a}\in \mathcal{I}$ as a product of $k$ (not necessarily distinct) ideals.
\item If $(C)$ is a statement which is either true or false, then we write
$$
1_{(C)} \coloneqq \begin{cases} 1 \mbox{ if the } (C) \mbox{ is true,}\\ 0 \mbox{ otherwise.} \end{cases}
$$
If $M$ is a set, then we write
$$
1_{M}(x) \coloneqq 1_{(x\in M)} = \begin{cases} 1 \mbox{ if } x\in M,\\ 0 \mbox{ otherwise.} \end{cases}
$$
\item We write $r\sim R$ to indicate that $R\le r< R_1$ for some $R_1\in (R,2R]$. 
\end{itemize}

\subsection{The sieve result}
In~\cite[Theorem~3.1]{BT}, Harman's asymptotic sieve~\cite[Lemma~2]{harman1996on-the-distribu} was extended to imaginary quadratic fields of class number one 
by the first- and third-named 
authors, with an additional weight function included. In~\cite[Theorem~12]{BM}, this weighted sieve was further generalized by the first- 
and second-named authors into a version with ideals in arbitrary
quadratic number fields. Now we establish a weighted version of Harman's asymptotic sieve with ideals in general number fields (without restriction on the class number). As in~\cite[section~13]{BM}, 
we follow closely the lines of~\cite[section~7]{BT}. There is only one small technical issue which makes things slightly more complicated: Different prime ideals may have equal norms. This is handled by a certain decomposition in~\eqref{primedec1} and~\eqref{primedec2} below.  

First, we state a basic lemma---sometimes dubbed `cosmetic surgery'~\cite{harman1996on-the-distribu}---which will be used to separate summations over variables whose sizes are linked. 

\begin{Lemma} \label{disentangle}
For any two distinct real numbers $\rho,\gamma >0 $ and $T\geq1$ one has 
\begin{align*}
   \Bigg|1_{\gamma<\rho}-\frac{1}{\pi}\int_{-T}^{T}e^{i\gamma t} \frac{\sin(\rho t)}{t}\Bigg| \ll \frac{1}{T|\gamma-\rho|} ,
\end{align*}
where the implied constant is absolute.
\end{Lemma}
\begin{proof}
See, for instance,~\cite[Lemma~2.2]{harman2007primedetectingsieves}.
\end{proof}

Our main result in this section is the following. 

\begin{Theorem}[Harman's asymptotic sieve for $\mathcal{I}$] \label{asymp} Let $x\geq3$ be real and let $\omega,\Tilde{\omega}:\mathcal{I}\Rightarrow \mathbb{R}_{\ge 0}$ 
be two 
functions such that, for both $w=\omega$ and $w=\Tilde{\omega}$,
  \begin{align} \label{(2.1)}
  \sum_{\mathfrak{n}\in \mathcal{I}} d_5(\mathfrak{n})w(\mathfrak{n})\leq X 
\end{align}
for $X\geq 1$. Suppose further one has $Y>1,0\le \iota\le \mu<1,0<\kappa\leq 1/2$ and $M\in [x^\mu,x)$ such that 
for any sequences $(a_\mathfrak{a})_{\mathfrak{a} \in \mathcal{I}}$,$(b_\mathfrak{b})_{\mathfrak{b} \in \mathcal{I}}$ of complex numbers with 
$|a_\mathfrak{a}|\leq d_3(\mathfrak{a})$ and $|b_\mathfrak{b}|\leq d_3(\mathfrak{b})$, the inequalities
\begin{align} \label{(2.2)}
    |\mathop{\sum\sum}\limits_{\substack{\mathfrak{a},\mathfrak{b} \in \mathcal{I}\\ \mathcal{N}(\mathfrak{a})\le M}} a_\mathfrak{a} (\omega(\mathfrak{ab})-\Tilde{\omega}(\mathfrak{ab}))| \leq Y
\end{align}
and
\begin{align} \label{(2.3)}
    |\mathop{\sum\sum}\limits_{\substack{\mathfrak{a},\mathfrak{b} \in \mathcal{I}\\ 
    x^{\mu-\iota}\le \mathcal{N}(\mathfrak{a})\le x^{\mu+\kappa}}} a_\mathfrak{a} b_\mathfrak{b} (\omega(\mathfrak{ab})-\Tilde{\omega}(\mathfrak{ab}))| \leq Y
\end{align}
hold. Also suppose that $|c_{\mathfrak{r}}|\le 1$ and $R>1/2$.  
If $R>x^{\mu+\kappa}/2$, then, in addition, suppose that $2R\le x^{1-\mu}\le M$ and 
\begin{equation} \label{smallandlarge}
\sum\limits_{\substack{\mathfrak{a}\in \mathcal{I}\\ \mathcal{N}(\mathfrak{a})\not\in [x^{1-\iota},x]}} d_5(\mathfrak{a})w(\mathfrak{a}) \leq Y.
\end{equation}
Then
\begin{equation} \label{Phi}
    \left|\sum\limits_{\mathcal{N}(\mathfrak{r})\sim R} c_{\mathfrak{r}}\Phi_{\mathfrak{r}}(\omega,\mathfrak{q})-
    \sum\limits_{\mathcal{N}(\mathfrak{r})\sim R} c_{\mathfrak{r}}\Phi_{\mathfrak{r}}(\Tilde{\omega},\mathfrak{q})\right|\ll Y\log^3(xX)
\end{equation}
for any prime ideal $\mathfrak{q}\in \mathbb{P}$ such that $\mathcal{N}(\mathfrak{p})< x^{\kappa}$ whenever $\mathfrak{p} \prec \mathfrak{q}$
and 
\begin{equation} \label{S}
    \left|\sum\limits_{\mathcal{N}(\mathfrak{r})\sim R} c_{\mathfrak{r}}S_{\mathfrak{r}}(\omega,x^{\kappa})-
    \sum\limits_{\mathcal{N}(\mathfrak{r})\sim R} c_{\mathfrak{r}}S_{\mathfrak{r}}(\Tilde{\omega},x^{\kappa})\right|\ll Y\log^3(xX),
\end{equation}
where the implied constants may depend on the field $\mathbb{K}$. 
\end{Theorem}

In applications, an asymptotic is already known for one of the sums 
$$
\sum\limits_{\mathcal{N}(\mathfrak{r})\sim R} c_{\mathfrak{r}}S_{\mathfrak{r}}(\omega,x^{\kappa}), \quad 
    \sum\limits_{\mathcal{N}(\mathfrak{r})\sim R} c_{\mathfrak{r}}S_{\mathfrak{r}}(\Tilde{\omega},x^{\kappa})
$$
and so an asymptotic formula can be deduced from \eqref{S} for the other sum. Theorem \ref{asymp} reduces the problem to establishing suitable estimates for bilinear sums. The special bilinear sums in \eqref{(2.2)} are commonly known as Type I sums. Here the summation over one of the variables is smooth (in our case, over the variable $\mathfrak{b}$). The general bilinear sums in \eqref{(2.3)} are usually referred to as Type II sums. Here we have general complex coefficients $a_{\mathfrak{a}}$ and $b_{\mathfrak{b}}$ for both variables. The situation is similar in applications of Vaughan's identity, where one is also led to estimating Type I and Type II sums. We proceed to the proof of Theorem \ref{asymp}. 

\begin{proof}
Estimate~\eqref{S} follows immediately from estimate~\eqref{Phi} since
$$
S_{\mathfrak{r}}(w,x^{\kappa})=\Phi_{\mathfrak{r}}\left(w,\mathfrak{Q}\left(x^{\kappa}\right)\right)
$$
for $w=\omega,\tilde{\omega}$. 
(Recall that $\mathfrak{Q}\left(x^{\kappa}\right)$ is the smallest prime ideal 
$\mathfrak{q}$ with respect to the ordering $\prec$
with norm greater or equal $x^{\kappa}$.) It remains to prove~\eqref{Phi}.

We first assume that $R\le x^{\mu}/2$. We define the M\"obius $\mu$ function for non-zero ideals by  
\begin{align*}
 \mu(\mathfrak{a})\coloneqq 
       \begin{cases}
       (-1)^k &\quad \text{if} \quad \mathfrak{a}=\mathfrak{p}_1\cdots \mathfrak{p}_k, \mbox{ where } \mathfrak{p}_1,...,\mathfrak{p}_k \mbox{ are distinct prime ideals,} \\
       0 &\quad\text{otherwise.}
     \end{cases}
\end{align*}
Then 
\begin{align} \label{(2.4)}
    \Phi_{\mathfrak{r}}(w,\mathfrak{q}) &
    =\sum_{\mathfrak{a}\in \mathcal{I}} w(\mathfrak{ar})\sum_{\substack{\mathfrak{d}|\Pi(\mathfrak{q})\\ \mathfrak{d}|\mathfrak{a}}}
    \mu(\mathfrak{d})
    = \sum_{\mathfrak{d}|\Pi(\mathfrak{q})} \mu(\mathfrak{d})\sum_{\mathfrak{b}\in\mathcal{I}}w(\mathfrak{bdr}).
\end{align}
Let 
\begin{align} \label{(2.5)}
\Delta(\mathfrak{c})=\sum_{\mathfrak{b}\in \mathcal{I}}(\omega(\mathfrak{bc})-\Tilde{\omega}(\mathfrak{bc})). 
\end{align}
Applying~\eqref{(2.4)} for $w=\omega$ and $w=\Tilde{\omega}$ yields
\begin{align} \label{(2.6)}
\Phi_{\mathfrak{r}}(\omega,\mathfrak{q})-\Phi_{\mathfrak{r}}(\Tilde{\omega},\mathfrak{q})=\bigg\{\sum_{\substack{\mathfrak{d}|\Pi(\mathfrak{q})\\ \mathcal{N}(\mathfrak{dr})<M}} + \sum_{\substack{\mathfrak{d}|\Pi(\mathfrak{q})\\ \mathcal{N}(\mathfrak{dr})\geq M}}\bigg\}
\mu(\mathfrak{d})\Delta(\mathfrak{dr})
\eqqcolon \Phi_{\mathfrak{r}}^{\sharp} + \Phi_{\mathfrak{r}}^{\flat}. 
\end{align}
Set 
$$
\Phi^{\sharp}\coloneqq\sum\limits_{\mathcal{N}(\mathfrak{r})\sim R} c_{\mathfrak{r}}\Phi_{\mathfrak{r}}^{\sharp} 
$$
and 
$$
\Phi^{\flat}\coloneqq\sum\limits_{\mathcal{N}(\mathfrak{r})\sim R} c_{\mathfrak{r}}\Phi_{\mathfrak{r}}^{\flat}.
$$
Using~\eqref{(2.2)} with 
$$
a_{\mathfrak{a}}=\sum\limits_{\substack{\mathfrak{r}|\mathfrak{a}\\ \mathcal{N}(\mathfrak{r})\sim R}} c_{\mathfrak{r}}
\mu\left(\frac{\mathfrak{a}}{\mathfrak{r}}\right)1_{\mathfrak{(\mathfrak{a}/\mathfrak{r})}|\Pi(\mathfrak{q})},
$$
we deduce that
$$
|\Phi^{\sharp}| \leq Y.
$$
Therefore, to prove the theorem, we need to show that 
\begin{align} \label{(2.7)}
|\Phi^{\flat}|\ll Y\log^3(xX). 
\end{align} 

The next step is to arrange $\Phi^{\flat}$ into subsums according to the sizes of the prime factors in $\mathfrak{d}$ (where $\mathfrak{d}$ is the summation variable from~\eqref{(2.6)}).
Take $g:\mathcal{I}\Rightarrow\mathbb{C}$ to be any function. We may group the terms of the sum
\begin{align*}
\Phi_{\mathfrak{r}}=\sum_{\mathfrak{d}|\Pi(\mathfrak{q})}\mu(\mathfrak{d})g(\mathfrak{dr})
\end{align*}
according to the largest factor $\mathfrak{p}_1$  of $\mathfrak{d}$ (w.r.t. $\prec$), getting the identity
\begin{align} \label{(2.8)}
\Phi_{\mathfrak{r}} = g(\mathfrak{r})-\sum_{\mathfrak{p}_1\in \mathbb{P}(\mathfrak{q})}\sum_{\mathfrak{d}|\Pi(\mathfrak{p}_1)} 
\mu(\mathfrak{d})g(\mathfrak{p}_1\mathfrak{dr}).
\end{align}
Again for the part $\sum_{\mathfrak{d}|\Pi(\mathfrak{p}_1)} \mu(\mathfrak{d})g(\mathfrak{p}_1\mathfrak{dr})$, we have 
\begin{align} \label{(2.9)}
    \sum_{\mathfrak{d}|\Pi(\mathfrak{p}_1)} \mu(\mathfrak{d})g(\mathfrak{p}_1\mathfrak{dr})= g(\mathfrak{p}_1\mathfrak{r})-\sum_{\mathfrak{p}_2\prec\mathfrak{p}_1}
    \sum_{\mathfrak{d}|\Pi(\mathfrak{p}_2)} \mu(\mathfrak{d})g(\mathfrak{p}_1\mathfrak{p}_2\mathfrak{dr}). 
\end{align}
Minding the innermost sum on the right-hand side above, it is obvious that the above identity can be iterated if so desired. 
To describe for which sub-sums the iteration is beneficial, we let
\begin{align*}
    \mathbb{P}(\mathfrak{q})&
    = \{\mathfrak{p}_1\in \mathbb{P}(\mathfrak{q}):\mathcal{N}(\mathfrak{p}_1\mathfrak{r})\ge x^\mu\}\mathbin{\dot{\cup}}
    \{\mathfrak{p}_1\in \mathbb{P}(\mathfrak{q}):\mathcal{N}(\mathfrak{p}_1\mathfrak{r})< x^\mu\}\\&
    =\mathcal{P}_{\mathfrak{r}}(1)\mathbin{\dot{\cup}}\mathcal{Q}_{\mathfrak{r}}(1),\quad \mbox{say},
\end{align*}
and inductively for $s=2,3,...$,
\begin{align*}
  \mathcal{Q'}_{\mathfrak{r}}(s) &= \{(\mathfrak{p}_1,...,\mathfrak{p}_{s-1},\mathfrak{p}_s) \in {\mathbb{P}(\mathfrak{q})}^s:\mathfrak{p}_s \prec \mathfrak{p}_{s-1},(\mathfrak{p}_1,...,\mathfrak{p}_{s-1})\in 
  \mathcal{Q}_{\mathfrak{r}}(s-1)\}\\&
  =\mathcal{P}_{\mathfrak{r}}(s)\mathbin{\dot{\cup}}\mathcal{Q}_{\mathfrak{r}}(s),
\end{align*}
where
\begin{align*}
  \mathcal{P}_{\mathfrak{r}}(s) &= \{(\mathfrak{p}_1,...,\mathfrak{p}_{s-1},\mathfrak{p}_s)\in\mathcal{Q'}_{\mathfrak{r}}(s):\mathcal{N}(\mathfrak{p}_1\mathfrak{p}_2\
  \cdots \mathfrak{p}_s\mathfrak{r})\ge x^\mu\},\\
  \mathcal{Q}_{\mathfrak{r}}(s) &= \{(\mathfrak{p}_1,...,\mathfrak{p}_{s-1},\mathfrak{p}_s)\in\mathcal{Q'}_{\mathfrak{r}}(s):\mathcal{N}(\mathfrak{p}_1\mathfrak{p}_2
  \cdots \mathfrak{p}_s\mathfrak{r})< x^\mu\}.
\end{align*}
Assuming that $g$ vanishes on arguments $\mathfrak{a}$ with $\mathcal{N}(\mathfrak{a})< x^\mu$, and on applying~\eqref{(2.8)} and~\eqref{(2.9)} we have 
\begin{align*}
\Phi_{\mathfrak{r}} &
 =-\bigg(\sum_{\mathfrak{p}_1\in\mathcal{P}_{\mathfrak{r}}(1)}+\sum_{\mathfrak{p}_1\in\mathcal{Q}_{\mathfrak{r}}(1)}\bigg)
\sum_{\mathfrak{d}|\Pi(\mathfrak{p}_1)}\mu(\mathfrak{d})g(\mathfrak{p}_1\mathfrak{dr})\\ &
  = -\sum_{\mathfrak{p}_1\in\mathcal{P}_{\mathfrak{r}}(1)}\sum_{\mathfrak{d}|\Pi(\mathfrak{p}_1)}\mu(\mathfrak{d})g(\mathfrak{p}_1\mathfrak{dr})+
  \sum_{(\mathfrak{p}_1,\mathfrak{p}_2)\in \mathcal{P}_{\mathfrak{r}}(2)}\sum_{\mathfrak{d}|\Pi(\mathfrak{p}_2)}\mu(\mathfrak{d})g(\mathfrak{p}_1\mathfrak{p}_2\mathfrak{dr})\\
  & \quad \quad \quad \quad + \sum_{(\mathfrak{p}_1,\mathfrak{p}_2)\in \mathcal{Q}_{\mathfrak{r}}(2)}\sum_{\mathfrak{d}|\Pi(\mathfrak{p}_2)}\mu(\mathfrak{d})g(\mathfrak{p}_1
  \mathfrak{p}_2\mathfrak{dr}),
\end{align*}
where we note that $g(\mathfrak{r})=0$ in~\eqref{(2.8)} because of our assumption $\mathcal{N}(\mathfrak{r})< 2R\le x^{\mu}$. 

On iterating this process, always applying~\eqref{(2.9)} to the $\mathcal{Q}$-part, it transpires that 
\begin{align*}
   \Phi_{\mathfrak{r}} &
   =\sum_{s\leq t}(-1)^s\sum_{(\mathfrak{p}_1,...,\mathfrak{p}_s)\in\mathcal{P}_{\mathfrak{r}}(s)}\sum_{\mathfrak{d}|\Pi(\mathfrak{p}_s)}
   \mu(\mathfrak{d})g(\mathfrak{p}_1\mathfrak{p}_2\cdots \mathfrak{p}_s\mathfrak{dr})\\
   & \quad +(-1)^t\sum_{(\mathfrak{p}_1,...,\mathfrak{p}_t)\in\mathcal{Q}_{\mathfrak{r}}(t)}\sum_{\mathfrak{d}|\Pi(\mathfrak{p}_t)}\mu(\mathfrak{d})g(\mathfrak{p}_1\mathfrak{p}_2\cdots \mathfrak{p}_t\mathfrak{dr})
\end{align*}
for any $t\in \mathbb{N}$. Since the product of $t$ prime ideals has norm greater than or equal to $2^t$ we have
\begin{align*}
    \mathcal{Q}_{\mathfrak{r}}(t)=\emptyset \quad \mbox{for} \quad t>\frac{\mu}{\log 2}\log x.
\end{align*}
Hence,
\begin{align*}
    \Phi_{\mathfrak{r}}=\sum_{s\leq t}(-1)^s \sum_{(\mathfrak{p}_1,...,\mathfrak{p}_s)\in\mathcal{P}_{\mathfrak{r}}(s)}\sum_{\mathfrak{d}|\Pi(\mathfrak{p}_s)}
    \mu(\mathfrak{d})g(\mathfrak{p}_1\mathfrak{p}_2\cdots \mathfrak{p}_s\mathfrak{dr})
\end{align*}
for 
\begin{align} \label{(2.10)}
    t=\Bigl\lfloor{\frac{\log x}{\log 2}}\Bigr\rfloor+1 \ll \log x.
\end{align} 
We apply this to $\Phi_{\mathfrak{r}}^{\flat}$ with 
\begin{equation} \label{gchoice}
g(\mathfrak{a})=\Delta(\mathfrak{a})1_{\{{\mathcal{N}(\mathfrak{a})\geq M}\}}.
\end{equation}
Note that since 
$M\ge x^\mu$, we have $g(\mathfrak{a})=0$ for all $\mathcal{N}(\mathfrak{a})< x^\mu$, as was assumed in the above arguments. Thus,
\begin{align} \label{(2.11)}
    \Phi_{\mathfrak{r}}^{\flat}=\sum_{s\leq t}  (-1)^s\Phi_{\mathfrak{r}}^{\flat}(s),
\end{align}
where 
\begin{align*}
     \Phi_{\mathfrak{r}}^{\flat}(s)=\sum_{\substack{(\mathfrak{p}_1,...,\mathfrak{p}_s)\in\mathcal{P}_{\mathfrak{r}}(s) \\ \mathfrak{a}=\mathfrak{p}_1\cdots \mathfrak{p}_s}}
     \sum_{\substack{\mathfrak{d}|\Pi(\mathfrak{p}_s)\\ \mathcal{N}(\mathfrak{adr})\geq M}} \mu(\mathfrak{d})\Delta(\mathfrak{adr}). 
\end{align*}
Another application of~\eqref{(2.9)} gives 
\begin{equation} \label{(2.12)}
\begin{aligned}
 \Phi_{\mathfrak{r}}^{\flat}(s)&
 =\sum_{\substack{(\mathfrak{p}_1,...,\mathfrak{p}_s)\in\mathcal{P}_{\mathfrak{r}}(s) \\ 
 \mathfrak{a}=\mathfrak{p}_1\cdots \mathfrak{p}_s\\ \mathcal{N}(\mathfrak{ar})\geq M}} \Delta(\mathfrak{ar})-
 \sum_{\substack{(\mathfrak{p}_1,...,\mathfrak{p}_s)\in\mathcal{P}_{\mathfrak{r}}(s) \\ \mathfrak{a}=\mathfrak{p}_1\cdots \mathfrak{p}_s}} 
 \sum_{\mathfrak{p}\prec \mathfrak{p}_s} \sum_{\substack{\mathfrak{d}|\Pi(\mathfrak{p)} \\ \mathcal{N}(\mathfrak{apdr})\geq M}} 
 \mu(\mathfrak{d})\Delta(\mathfrak{apdr}) \\&
 \eqqcolon \Phi_{\mathfrak{r},1}^{\flat}(s)-\Phi_{\mathfrak{r},2}^{\flat}(s).
\end{aligned}
\end{equation}
Given $\mathfrak{a}=\mathfrak{p}_1\cdots \mathfrak{p}_{s-1}\mathfrak{p}_s$ with 
\begin{align*}
(\mathfrak{p}_1,...,\mathfrak{p}_{s-1},\mathfrak{p}_s)\in \mathcal{P}_{\mathfrak{r}}(s) \quad \mbox{and} \quad (\mathfrak{p}_1,...,\mathfrak{p}_{s-1})\in \mathcal{Q}_{\mathfrak{r}}(s-1) \mbox{ if } s\ge 2,
\end{align*}
and noting that $\mathcal{N}(\mathfrak{p}_s)\leq \mathcal{N}(\mathfrak{p}_1)<x^\kappa $ (recall that $\mathcal{N}(\mathfrak{p})<x^{\kappa}$ for all prime ideals $\mathfrak{p}\prec \mathfrak{q}$), we have
\begin{align} \label{correctsize}
  x^\mu \le M\le \mathcal{N}(\mathfrak{ar})=\mathcal{N}(\mathfrak{p}_1\cdots \mathfrak{p}_{s-1}\mathfrak{r})\mathcal{N}(\mathfrak{p_s})<x^{\mu+\kappa}.
\end{align}
This works also in the case when $s=1$ since in this case the product $\mathfrak{p}_1\cdots \mathfrak{p}_{s-1}$ is empty and we have assumed that $\mathcal{N}(\mathfrak{r})<2R\le x^{\mu}$.    
Using this, we find that $\Phi_{\mathfrak{r},1}^{\flat}(s)$ can be expressed as 
\begin{align} \label{S1sum}
 \Phi_{\mathfrak{r},1}^{\flat}(s)=
 \mathop{\sum\sum}\limits_{\mathfrak{a},\mathfrak{b} \in \mathcal{I}} a_{\mathfrak{r},\mathfrak{a}} (\omega(\mathfrak{abr})-\Tilde{\omega}(\mathfrak{abr})) 
\end{align}
with coefficients
\begin{align*}
a_{\mathfrak{r},\mathfrak{a}}=1_{\{\mathcal{N}(\mathfrak{ar})\geq M\}} 1_{\{\mathfrak{p}_1\cdots \mathfrak{p}_s: (\mathfrak{p}_1,...,\mathfrak{p}_s)\in \mathcal{P}_{\mathfrak{r}}(s)\}}(\mathfrak{a})
\end{align*}
only supported on $(\mathfrak{r},\mathfrak{a})$ with $x^\mu\le \mathcal{N}(\mathfrak{ar})< x^{\mu+\kappa}$. Hence setting 
$\mathfrak{A}=\mathfrak{ar}$ and using~\eqref{(2.3)}, we have 
\begin{align} \label{(2.13)}
\left|\sum\limits_{\mathcal{N}(\mathfrak{r})\sim R} c_{\mathfrak{r}}\Phi_{\mathfrak{r},1}^{\flat}(s)\right|=\left|\sum\limits_{\substack{\mathfrak{A},\mathfrak{b}\\
x^{\mu}\le\mathcal{N}(\mathfrak{A})< x^{\mu+\kappa}}} a_{\mathfrak{a}} (\omega(\mathfrak{A}\mathfrak{b})-\Tilde{\omega}(\mathfrak{A}\mathfrak{b}))\right|\leq Y, 
\end{align}
where 
$$
a_{\mathfrak{A}}=\sum\limits_{\substack{\mathcal{N}(\mathfrak{r})\sim R\\ \mathfrak{r}|\mathfrak{A}}} 
c_{\mathfrak{r}}a_{\mathfrak{r},\mathfrak{A}/\mathfrak{r}}.
$$

Moving on to $\Phi_{\mathfrak{r},2}^{\flat}(s)$, we expand the definition~\eqref{(2.5)} of $\Delta$, getting
\begin{align*}
    \Phi_{\mathfrak{r},2}^{\flat}(s)=\Phi_{\mathfrak{r},2}^{\flat}(s,\omega)-\Phi_{\mathfrak{r},2}^{\flat}(s,\Tilde{\omega}),
\end{align*}
where
\begin{equation}\label{S2sum}
\begin{aligned}
 \Phi_{\mathfrak{r},2}^{\flat}(s,w)&\coloneqq\sum_{\substack{(\mathfrak{p}_1,...,\mathfrak{p}_s)\in\mathcal{P}_{\mathfrak{r}}(s) \\ \mathfrak{a}=\mathfrak{p}_1\cdots \mathfrak{p}_s}} 
 \sum_{\mathfrak{p}\prec \mathfrak{p}_s} \sum_{\substack{\mathfrak{d}|\Pi(\mathfrak{p)} \\ \mathcal{N}(\mathfrak{apdr})\geq M}} \mu(\mathfrak{d})
 \sum_{\mathfrak{n}\in\mathcal{I}}w(\mathfrak{anpdr})\\&
 =\sum_{\substack{(\mathfrak{p}_1,...,\mathfrak{p}_s)\in\mathcal{P}_{\mathfrak{r}}(s) \\ \mathfrak{a}=\mathfrak{p}_1\cdots \mathfrak{p}_s}}\sum_{\mathfrak{b}\in\mathcal{I}} 
 \sum_{\mathfrak{p}\prec \mathfrak{p}_s}\mathop{\sum\sum}\limits_{\substack{\mathfrak{d}|\Pi(\mathfrak{p})\\ \mathfrak{npd=b} \\ \mathcal{N}(\mathfrak{apdr})\geq M}}
 \mu(\mathfrak{d})w(\mathfrak{abr}).
 \end{aligned}
\end{equation}
 In order to apply~\eqref{(2.3)}, we must disentangle the variables $\mathfrak{a}$ and $\mathfrak{n}$ in the above summation. To this end, we split
 \begin{align*}
   \sum_{\mathfrak{p}\prec \mathfrak{p}_s} = \sum_{\substack{\mathfrak{p}\prec \mathfrak{p}_s\\\mathcal{N}(\mathfrak{p})=
   \mathcal{N}(\mathfrak{p}_s)}}+\sum_{\substack{\mathfrak{p}\prec \mathfrak{p}_s\\\mathcal{N}(\mathfrak{p})<\mathcal{N}(\mathfrak{p}_s)}} 
 \end{align*}
to obtain a decomposition 
 \begin{align} \label{(2.15)}
 \Phi_{\mathfrak{r},2}^{\flat}(s,w)=\Phi_{\mathfrak{r},2}^{=}(s,w)+\Phi_{\mathfrak{r},2}^{<}(s,w), \quad \mbox{say}. 
 \end{align}
For $\Phi_{\mathfrak{r},2}^{<}(s,w)$ we have 
 \begin{align*}
 \Phi_{\mathfrak{r},2}^{<}(s,w)=\sum_{\substack{(\mathfrak{p}_1,...,\mathfrak{p}_s)\in\mathcal{P}_{\mathfrak{r}}(s) \\ 
 \mathfrak{a}=\mathfrak{p}_1\cdots \mathfrak{p}_s}}
 \sum_{\mathfrak{b}\in\mathcal{I}} \sum_{\mathfrak{p}\prec \mathfrak{p}_s}\mathop{\sum\sum}\limits_{\substack{\mathfrak{d}|\Pi(\mathfrak{p})\\ \mathfrak{npd=b}}} 
 \mu(\mathfrak{d})\chi_{\mathfrak{r}}(\mathfrak{a,d,p},\mathfrak{p}_s) w(\mathfrak{abr}),
 \end{align*}
where 
 \begin{align*}
    \chi_{\mathfrak{r}}(\mathfrak{a,d,p},\mathfrak{p}_s)=1_{\{\mathcal{N}(\mathfrak{apdr})\ge M\}}1_{\{\mathcal{N}(\mathfrak{p})<\mathcal{N}(\mathfrak{p}_s) \}},
 \end{align*}
and the sum $\Phi_{\mathfrak{r},2}^{=}(s,w)$ can be expressed similarly, but needs a little more care.

First note that, by ramification theory, for any fixed number, there are at most $[\mathbb{K}:\mathbb{Q}]$ distinct prime ideals $\mathfrak{p} \subseteq \mathcal{O}$ whose norm coincides with that number.
We use this fact to split the set $\mathcal{P}_{\mathfrak{r}}(s)$ into disjoint sets in the form
\begin{equation} \label{primedec1}
	\mathcal{P}_{\mathfrak{r}}(s)
	= \mathcal{P}^{0}_{\mathfrak{r}}(s) \cupdot \mathcal{P}^{1}_{\mathfrak{r}}(s) \cupdot \ldots \cupdot \mathcal{P}^{[\mathbb{K}:\mathbb{Q}]-1}_{\mathfrak{r}}(s),
\end{equation}
where $(\mathfrak{p}_1,\ldots,\mathfrak{p}_s)\in\mathcal{P}_{\mathfrak{r}}(s)$ belongs to $\mathcal{P}_{\mathfrak{r}}^{u}(s)$ if and only if there are exactly $u$ distinct prime ideals, all smaller than $\mathfrak{p}_s$ with respect to $\prec$, and having norm equal to $\Norm(\mathfrak{p}_s)$.
Furthermore, let
\begin{equation} \label{primedec2}
    \mathbb{P}_{\mathfrak{r},s}^u(\mathfrak{q})
    = \mathcal{P}^{0}_{\mathfrak{r}}(s) \cupdot \mathcal{P}^{1}_{\mathfrak{r}}(s) \cupdot \ldots \cupdot \mathcal{P}^{u-1}_{\mathfrak{r}}(s).
\end{equation}
Then
\begin{equation} \label{decomp}
	\Phi_{\mathfrak{r},2}^{=}(s,w)
	= \sum_{0\leq u < [\mathbb{K}:\QQ]}
	\sum_{\substack{ (\mathfrak{p}_1,\ldots,\mathfrak{p}_s) \in \mathcal{P}_{\mathfrak{r}}^u(s) \\ \mathfrak{a} = \mathfrak{p}_1\cdots\mathfrak{p}_s }}
	\sum_{\mathfrak{b} \in \mathcal{I}}
	\sum_{ \mathfrak{p} \in \mathbb{P}_{\mathfrak{r},s}^u(\mathfrak{q}) }
	\mathop{\sum\sum}_{\substack{ \mathfrak{d}|\Pi(\mathfrak{p}) \\ \mathfrak{npd} = \mathfrak{b} }}
	\mu(\mathfrak{d})
	\tilde{\chi}_{\mathfrak{r}}(\mathfrak{a}, \mathfrak{d}, \mathfrak{p}, \mathfrak{p}_s)
	w(\mathfrak{abr}),
\end{equation}
where we have put
\begin{gather}\label{(2.16)}
	\begin{aligned}
		\tilde{\chi}_{\mathfrak{r}}(\mathfrak{a}, \mathfrak{d}, \mathfrak{p}, \mathfrak{p}_s) &
		= 1_{\braces{ \Norm(\mathfrak{apdr}) \geq M }} 1_{\braces{ \Norm(\mathfrak{p}) = \Norm(\mathfrak{p}_s) }} \\ &
		= 1_{\braces{ \Norm(\mathfrak{apdr}) \geq M }} 1_{\braces{ \Norm(\mathfrak{p}) \leq \Norm(\mathfrak{p}_s) }} - \chi_{\mathfrak{r}}(\mathfrak{a}, \mathfrak{d}, \mathfrak{p}, \mathfrak{p}_s).
	\end{aligned}
\end{gather}

To separate summations, we now use Lemma~\ref{disentangle}.
We choose some real number $\rho$ with $|\rho|\leq 1/2$ such that $\{M+\rho\}=1/2$ and hence the condition 
$\mathcal{N}(\mathfrak{apdr})\geq M$ implies
\begin{align*}
    |\log \mathcal{N}(\mathfrak{apdr})- \log(M+\rho)|\geq \log\frac{x+1}{x+1/2}\geq \frac{1}{3x}.
\end{align*}
Therefore, Lemma~\ref{disentangle} shows that 
\begin{align*}
    1_{\{\mathcal{N}(\mathfrak{apdr})\geq M\}}=1-\frac{1}{\pi}\int_{-T}^{T}\mathcal{N}(\mathfrak{apdr})^{i\tau}\sin{(\tau\log(M+\rho))}\frac{d\tau}{\tau}+ 
    O\left(\frac{x}{T}\right)
\end{align*}
for every $T\geq 1$. Similarly,
\begin{align*}
    1_{\{\mathcal{N}(\mathfrak{p})<\mathcal{N}(\mathfrak{p}_s)\}} & =\frac{1}{\pi}\int_{-T}^{T}e^{it/2}e^{it\mathcal{N}(\mathfrak{p})}\sin{(t\mathcal{N}(\mathfrak{p}_s))}\frac{dt}{t}+
    O\left(\frac{1}{T}\right),\\
    1_{\{\mathcal{N}(\mathfrak{p})\le \mathcal{N}(\mathfrak{p}_s)\}} & =
    \frac{1}{\pi}\int_{-T}^{T}e^{-it/2}e^{it\mathcal{N}(\mathfrak{p})}\sin{(t\mathcal{N}(\mathfrak{p}_s))}\frac{dt}{t}+O\left(\frac{1}{T}\right).
\end{align*}
Thus, 
\begin{equation} \label{(2.17)}
\begin{aligned}
  \Phi_{\mathfrak{r},2}^{<}(s,w)&
  =\frac{1}{\pi}\int_{-T}^{T}\mathop{\sum\sum}\limits_{\substack{\mathfrak{a},\mathfrak{b}\in \mathcal{I}}}
  a_\mathfrak{a}(t)b_\mathfrak{b}(t)w(\mathfrak{abr})\frac{dt}{t}\\&- \frac{1}{\pi^2}\int_{-T}^{T}\int_{-T}^{T}\mathop{\sum\sum}
  \limits_{\substack{\mathfrak{a},\mathfrak{b}\in \mathcal{I}}}\mathcal{N}(\mathfrak{r})^{it}a_\mathfrak{a}(t,\tau)b_\mathfrak{b}(t,\tau)w(\mathfrak{abr})
  \frac{d\tau}{\tau}\frac{dt}{t}\\ 
   & + O\Bigg(\Bigg(\frac{x}{T}+\frac{1}{T}\int_{-T}^{T}|\sin{(\tau\log(M+\rho))}|\frac{d\tau}{\tau}\Bigg)\times\\ &  \quad\qquad
  \Bigg(\sum_{\substack{(\mathfrak{p}_1,...,\mathfrak{p}_s)\in\mathcal{P}_{\mathfrak{r}}(s) \\ \mathfrak{a}=\mathfrak{p}_1\cdots \mathfrak{p}_s}}
  \sum_{\mathfrak{b}\in\mathcal{I}} \sum_{\mathfrak{p}\prec \mathfrak{p}_s}\mathop{\sum\sum}\limits_{\substack{\mathfrak{d}|\Pi(\mathfrak{p})\\ 
  \mathfrak{npd=b}}}  w(\mathfrak{abr})\Bigg)\Bigg),
\end{aligned}
\end{equation}
with coefficients 
\begin{equation} \label{(2.18)}
\begin{split}
    a_{\mathfrak{a}}(t)&=
       \begin{cases}
       \sin{(t\mathcal{N}(\mathfrak{p}_s))} \quad \mbox{if } \exists (\mathfrak{p}_1,...,\mathfrak{p}_s)\in\mathcal{P}_{\mathfrak{r}}(s) \mbox{ such that } 
       \mathfrak{a}=\mathfrak{p}_1 \cdots \mathfrak{p}_s,   \\
       0 \quad \text{otherwise,}
     \end{cases}\\
    b_\mathfrak{b}(t)&=\sum_{\mathfrak{p}\in\mathbb{P}(\mathfrak{q})}\mathop{\sum\sum}\limits_{\substack{\mathfrak{d}|\Pi(\mathfrak{p})\\ \mathfrak{npd=b}}}e^{it/2}e^{it\mathcal{N}(\mathfrak{p})}\mu(\mathfrak{d}),\\
    a_{\mathfrak{a}}(t,\tau)&= a_{\mathfrak{a}}(t)\mathcal{N}(\mathfrak{a})^{i\tau}\sin{(\tau\log(M+\rho))},\\
    b_{\mathfrak{n}}(t,\tau)&=\sum_{\mathfrak{p}\in\mathbb{P}(\mathfrak{q}) }\mathop{\sum\sum}\limits_{\substack{\mathfrak{d}|\Pi(\mathfrak{p})\\ 
    \mathfrak{npd=b}}}e^{\frac{it}{2}}e^{it\mathcal{N}(\mathfrak{p})}\mu(\mathfrak{d})\mathcal{N}(\mathfrak{pd})^{i\tau}. 
\end{split}
\end{equation}

We proceed by gathering some intermediate information before applying~\eqref{(2.3)}: Clearly,
\begin{align*}
     |b_\mathfrak{b}(t)|,|b_{\mathfrak{b}}(t,\tau)|\leq d_2(\mathfrak{b}).
\end{align*}
For the other coefficients we always have
\begin{align*}
     |a_\mathfrak{a}(t)|,|a_{\mathfrak{a}}(t,\tau)|\leq 1,
\end{align*}
yet if $t$ and $\tau$ are small, we can do better: If $|t|\leq x^{-1/2}$ and $|\tau|\leq 1/\log(x+1/2)$, then
\begin{align} \label{(2.19)}
    |a_\mathfrak{a}(t)|\leq\sqrt{x}|t|, \quad |a_{\mathfrak{a}}(t,\tau)|\leq\sqrt{x}|t\tau|\log\left(x+1/2\right). 
\end{align}
In view of this, we must deal with functions $f:\mathbb{R} \times (1,\infty)\rightarrow \mathbb{R}$ of the shape
\begin{align*}
    f(t,\delta)=
       \begin{cases}
       \delta |t| \quad if\quad |t|\leq \delta^{-1},\\ 
       1 \quad \text{otherwise}\\
     \end{cases}
\end{align*}
and their integrals
\begin{align} \label{(2.20)}
    \int_{-T}^{T} f(t,\delta)\frac{dt}{|t|} \ll \delta \int_{0}^{\delta^{-1}} dt + \Bigg|\int_{\delta^{-1}}^T \frac{dt}{t}\Bigg|\ll 1+|\log{(T\delta)}|. 
\end{align}
Lastly, we note that by~\eqref{(2.1)}, we have 
\begin{align} \label{(2.21)}
    \left|\sum\limits_{\mathcal{N}(\mathfrak{r})\sim R}
    \sum_{\substack{(\mathfrak{p}_1,...,\mathfrak{p}_s)\in\mathcal{P}_{\mathfrak{r}}(s) \\ \mathfrak{a}=\mathfrak{p}_1\cdots 
    \mathfrak{p}_s}}\sum_{\mathfrak{b}\in\mathcal{I}} \sum_{\mathfrak{p}\prec 
    \mathfrak{p}_s}\mathop{\sum\sum}\limits_{\substack{\mathfrak{d}|\Pi(\mathfrak{p})\\ \mathfrak{npd=b}}}  w(\mathfrak{abr})\right| \le
   \sum_{\mathfrak{a}\in \mathcal{I}} d_5(\mathfrak{a})w(\mathfrak{a}) \le X. 
\end{align}
Gathering all information we got so far, we may derive a bound for 
\begin{align*}
    \mathcal{E}^{<} \coloneqq \left|\sum_{\mathcal{N}(\mathfrak{r})\sim R} c_{\mathfrak{r}}\Phi_{\mathfrak{r},2}^{<}(s,\omega)-
    \sum_{\mathcal{N}(\mathfrak{r})\sim R} c_{\mathfrak{r}}\Phi_{\mathfrak{r},2}^{<}(s,\Tilde{\omega})\right|
\end{align*}
as follows: after applying~\eqref{(2.17)} with 
$w=\omega$ and $w=\Tilde{\omega}$, the $O$-terms are treated directly with~\eqref{(2.21)} and~\eqref{(2.20)},
whereas for the rest we apply~\eqref{(2.3)} after summing over $\mathfrak{r}$ and merging $\mathfrak{a}$ and $\mathfrak{r}$ into $\mathfrak{A}=\mathfrak{ar}$. 
Here it is important to use~\eqref{(2.19)} for small $|t|$ and $|\tau|$ 
first - prior to applying~\eqref{(2.3)} - and~\eqref{(2.20)} then bounds the integrals. Therefore, after some computations, we infer 
\begin{align} \label{(2.22)}
     \mathcal{E}^{<}\ll &Y\log(Tx)(1+\log(T\log(x+1/2)))+
      XT^{-1}(x+\log(T\log(x+1/2))). 
\end{align}
Of course, the same arguments also apply to  
$$
\mathcal{E}^{=} \coloneqq \left|\sum\limits_{\mathcal{N}(\mathfrak{r})\sim R} c_r\Phi_{\mathfrak{r},2}^{=}(s,\omega)-\sum\limits_{\mathcal{N}(\mathfrak{r})\sim R} 
c_r\Phi_{\mathfrak{r},2}^{=}(s,\tilde{\omega})\right|.
$$
In view of~\eqref{(2.16)} we have to apply them twice, but in both cases the coefficients corresponding to~\eqref{(2.18)} 
obey the same bounds we used to derive~\eqref{(2.22)}, with the only difference that the implied constant may now depend on $[\mathbb{K}:\mathbb{Q}]$. (Note the decomposition into $[\mathbb{K}:\mathbb{Q}]$ subsums in~\eqref{decomp}.) Consequently, \eqref{(2.22)} also holds with $\mathcal{E}^{=}$ in place of 
$\mathcal{E}^{<}$. In total, recalling~\eqref{(2.12)}, \eqref{(2.13)} and~\eqref{(2.15)}, we have 
\begin{align*}
  \left| \sum\limits_{\mathcal{N}(\mathfrak{r})\sim R} c_{\mathfrak{r}} \Phi_{\mathfrak{r}}^{\flat}(s)\right| \ll  Y + \mbox{the bound from~\eqref{(2.22)}}
\end{align*}
and it transpires that choosing $T=xX$ suffices to yield a bound of $\ll Y\log^2(xX)$. On plugging this into~\eqref{(2.11)} and recalling~\eqref{(2.10)}, we infer 
\eqref{(2.7)}. This completes the proof for the case when $R\le x^{\mu}/2$. 

If $R>x^{\mu}/2$, then the iteration process described above terminates at $s=1$ since the set $\mathcal{Q}_{\mathfrak{r}}(1)$ is necessarily 
empty. In this case, the argument in~\eqref{correctsize} doesn't work anymore (for $s=1$). However, if $x^{\mu}/2<R\le x^{\mu+\kappa}/2$, then 
$x^{\mu}/2\le \mathcal{N}(\mathfrak{r})
< x^{\mu+\kappa}$, and so~\eqref{(2.3)} can be applied directly with $\mathfrak{a}=\mathfrak{r}$. 

It remains to consider the case when $x^{\mu+\kappa}/2<R\le x^{1-\mu}/2\le M/2$.
Here we start with the same procedure leading to~\eqref{S1sum} and~\eqref{S2sum}. 
We note that our choice of $g$ in~\eqref{gchoice} ensures that again $g(\mathfrak{r})=0$
in~\eqref{(2.8)} because $\mathcal{N}(r)< 2R\le M$. Now we look at the norm of $\mathfrak{ab}$ occurring in the said sums. If $x^{1-\iota}\le \mathcal{N}(\mathfrak{abr})\le x$, then $\mathcal{N}(\mathfrak{ab})$ lies in the range $x^{\mu-\iota}\le \mathcal{N}(\mathfrak{ab})\le x^{\mu+\kappa}$ and hence we may apply~\eqref{(2.3)} with $\mathfrak{a}$
replaced by $\mathfrak{ab}$ and $\mathfrak{b}$ by $\mathfrak{r}$. Indeed, in this case have $\mathcal{N}(\mathfrak{ab})< x^{\mu+\kappa}$ because 
$\mathcal{N}(\mathfrak{a})=\mathcal{N}(\mathfrak{p_1})< x^{\kappa}$ (recall that the above iteration process
terminates at $s=1$) and $\mathcal{N}(\mathfrak{b})\le x/\mathcal{N}(\mathfrak{ar})\le x/M\le x^{\mu}$, 
and we have
$\mathcal{N}(\mathfrak{ab})\ge x^{1-\iota}/\mathcal{N}(\mathfrak{r})> x^{\mu-\iota}$. 
Hence, similar arguments as above apply, where $\mathfrak{a}$ is now grouped together with $\mathfrak{b}$ in place of $\mathfrak{r}$. We handle the remaining contributions of $\mathcal{N}(\mathfrak{abr})>x$ and $\mathcal{N}(\mathfrak{abr})<x^{1-\iota}$ by means of~\eqref{smallandlarge}. 
This completes the proof.
\end{proof}

\section{Asymptotic estimates for \texorpdfstring{$\Phi_{\mathfrak{r}}(W,\mathfrak{p})$}{Φ𝔯(W,𝔭)}}
For the derivation of a number field version of Harman's lower bound sieve, we need asymptotic estimates for $\Phi_{\mathfrak{r}}(W,\mathfrak{p})$, where 
$W :\mathcal{I}\rightarrow \mathbb{R}_{\ge 0}$ is a function satisfying suitable conditions. 
This is the content of this section. The basic lemma here is Buchstab's identity which in this context takes the following form.

\begin{Lemma} \label{Buchstab} For any $\mathfrak{p},\mathfrak{q}\in \mathbb{P}$ with $\mathfrak{p} \prec \mathfrak{q}$ and $\mathfrak{r}\in \mathcal{I}\setminus \{0\}$, we have
$$
\Phi_{\mathfrak{r}}(W,\mathfrak{p})=\Phi_{\mathfrak{r}}(W,\mathfrak{q})+\sum\limits_{\mathfrak{p}\preceq \mathfrak{s}\prec \mathfrak{q}} 
\Phi_{\mathfrak{rs}}(W,\mathfrak{s}).
$$
\end{Lemma}

\begin{proof}
By definition of $\Phi_{\mathfrak{r}}(W,\mathfrak{p})$, we may write
$$
\Phi_{\mathfrak{r}}(W,\mathfrak{p})=\sum\limits_{\substack{\mathfrak{a}\in \mathcal{I}\\ \mathfrak{p}\preceq \mathfrak{P}^{-}(\mathfrak{a})}} W(\mathfrak{ar}), 
$$
where $\mathfrak{P}^{-}(\mathfrak{a})$ is the smallest prime ideal divisor of $\mathfrak{a}$ with respect to the above order. It follows that
\begin{equation}
 \begin{split}
\Phi_{\mathfrak{r}}(W,\mathfrak{p}) &= \Phi_{\mathfrak{r}}(W,\mathfrak{q})+\sum\limits_{\mathfrak{p}\preceq \mathfrak{s}\prec \mathfrak{q}} 
\sum\limits_{\substack{\mathfrak{a}\in \mathcal{I}\\
\mathfrak{P}^{-}(\mathfrak{a})=\mathfrak{s}}} W(\mathfrak{ar})\\ &= \Phi_{\mathfrak{r}}(W,\mathfrak{q})+\sum\limits_{\mathfrak{p}\preceq \mathfrak{s}\prec \mathfrak{q}} 
\sum\limits_{\substack{\mathfrak{b}\in \mathcal{I}\\
\mathfrak{s} \preceq \mathfrak{P}^{-}(\mathfrak{b})}} W(\mathfrak{bsr})\\ &=
\Phi_{\mathfrak{r}}(W,\mathfrak{q})+\sum\limits_{\mathfrak{p}\preceq \mathfrak{s}\prec \mathfrak{q}} 
\Phi_{\mathfrak{rs}}(W,\mathfrak{s}),
\end{split}
\end{equation}
as claimed. 
\end{proof}

Now we proceed in two stages. To have a clear picture of how the arguments work, we first handle weights which are characteristic functions of the form
$$
W(\mathfrak{a})=1_{1\le \mathcal{N}(\mathfrak{a})\le N}
$$
along classical lines using Landau's prime ideal theorem. In the second stage, we extend our results to more ge\-neral, possibly smooth weight functions whose 
averages over prime ideals satisfy suitable asymptotics which replace the prime number theorem.       

\subsection{Characteristic functions}
The following is Landau's prime ideal theorem. 

\begin{Lemma} \label{LPIT}
We have 
$$
\sharp \mathbb{P}(z) = \int\limits_{2}^z \frac{dt}{\log t} + O\left(z\exp(-C\sqrt{\log z})\right) =\frac{z}{\log z}+O\left(\frac{z}{\log^2 z}\right)\quad \mbox{as } z\rightarrow \infty,
$$
where $C$ is a positive constant which may depend on the field $\mathbb{K}$. 
\end{Lemma}

\begin{proof} See, for instance,~\cite[Theorem~8.9]{VaughanMontgomery}.
\end{proof}

Using Lemmas~\ref{Buchstab} and~\ref{LPIT}, we shall derive the following asymptotic estimate for the number of ideals without prime ideal divisors of small norm. 

\begin{Proposition} \label{roughideals} Let $N\ge 2$. For $\mathfrak{a}\in \mathcal{I}$ set
\begin{equation} \label{wchar}
W(\mathfrak{a})=1_{1\le \mathcal{N}(\mathfrak{a})\le N}.
\end{equation}
Let $\beta>2>\alpha>1$, $\mathfrak{p}\in \mathbb{P}$ and $\mathfrak{r}\in \mathcal{I}\setminus \{0\}$ and set
$$
y\coloneqq\frac{N}{\mathcal{N}(\mathfrak{r})} \quad \mbox{and} \quad u\coloneqq\frac{\log y}{\log \mathcal{N}(\mathfrak{p})}.
$$
Assume that
$$
\mathcal{N}(\mathfrak{r})\le N/2 \quad \mbox{and} \quad \alpha\le u \le \beta.
$$
Then,
\begin{equation} \label{genasymp}
\Phi_{\mathfrak{r}}\left(W,\mathfrak{p}\right)= \mathcal{B}(u)
\cdot \frac{y}{\log \mathcal{N}(\mathfrak{p})}\cdot \left(1+O_{\alpha,\beta}\left(\frac{1}{\log y}\right)\right),
\end{equation}
where $\mathcal{B}(u)$ is the Buchstab function, the unique continuous solution of the system
\begin{equation*}
\begin{split}
\mathcal{B}(u)&= \frac{1}{u} \mbox{ for } 1\le u\le 2,\\
\frac{d}{du}(u\mathcal{B}(u))&= \mathcal{B}(u-1) \mbox{ for } u>2.
\end{split}
\end{equation*}
\end{Proposition}

\begin{proof}  This can be established along similar lines as in the basic case $\mathbb{K}=\mathbb{Q}$, as described in~\cite[section 6.2]{Tenenbaum}. 
We start with establishing an asymptotic formula for $\Phi_{\mathfrak{r}}\left(W,\mathfrak{p}\right)$ in the range $1<u\le 2$ and then iteratively extend this range by deriving an asymptotic formula for $k+1<u\le k+2$ from an asymptotic formula for $k<u\le k+1$, where $k$ is any positive integer. 

The condition $1<u\le 2$ is equivalent to $y^{1/2}\le \mathcal{N}(\mathfrak{p})< y$, in which case we have 
\begin{equation} \label{basecase}
\begin{split}
\Phi_{\mathfrak{r}}(W,\mathfrak{p})& = \sum_{\substack{1\le \mathcal{N}(\mathfrak{a})\le y\\(\mathfrak{a},\Pi(\mathfrak{p}))=1}} 1\\ 
&= \sharp \mathbb{P}(y)-\sharp \mathbb{P}(\mathfrak{p})+1\\
&= \sharp \mathbb{P}(y)-\sharp \mathbb{P}(\mathcal{N}(\mathfrak{p}))+O(1)
\\ &=
\frac{y}{\log y} + O\left(\frac{y}{\log^2 y}+\frac{\mathcal{N}(\mathfrak{p})}{\log \mathcal{N}(\mathfrak{p})}\right)
\end{split}
\end{equation}
using Lemma~\ref{LPIT}. Hence, if $1<\alpha\le u\le 2$, then
$$
\Phi_{\mathfrak{r}}(W,\mathfrak{p})=
\frac{y}{\log y} + O_{\alpha}\left(\frac{y}{\log^2 y}\right),
$$
and therefore~\eqref{genasymp} holds in this range.  

Next we turn to the range $2<u\le 3$ which corresponds to $y^{1/3}\le  \mathcal{N}(\mathfrak{p})< y^{1/2}$. If $z>0$, recall that $\mathfrak{Q}(z)$ denotes the smallest prime ideal $\mathfrak{q}$ with respect to the ordering $\prec$ such that $\mathcal{N}(\mathfrak{q})\ge z$. In this case, Lemma~\ref{Buchstab} with
$\mathfrak{q}=\mathfrak{Q}\left(y^{1/2}\right)$ gives
\begin{equation} \label{BSapp}
\Phi_{\mathfrak{r}}(W,\mathfrak{p})= \Phi_{\mathfrak{r}}\left(W,\mathfrak{Q}\left(y^{1/2}\right)\right)+\sum\limits_{\substack{\mathfrak{p}\preceq \mathfrak{s} \prec \mathfrak{Q}\left(y^{1/2}\right)}} 
\Phi_{\mathfrak{rs}}(W,\mathfrak{s}).
\end{equation}
Using~\eqref{basecase}, we have
\begin{equation} \label{firstterm}
\Phi_{\mathfrak{r}}\left(W,\mathfrak{Q}\left(y^{1/2}\right)\right)
=\frac{y}{\log y}+O\left(\frac{y}{\log^2 y}\right).
\end{equation}
Further, we set  
\begin{equation} \label{yu}
y'\coloneqq\frac{N}{\mathcal{N}(\mathfrak{rs})}=\frac{y}{\mathcal{N}(\mathfrak{s})} \quad \mbox{and} \quad 
u'\coloneqq\frac{\log y'}{\log \mathcal{N}(\mathfrak{s})}=\frac{\log y}{\log \mathcal{N}(\mathfrak{s})}-1
\end{equation}
and observe that
$$
1<u'\le 2
$$
if $\mathfrak{Q}\left(y^{1/3}\right) \preceq \mathfrak{p}\preceq \mathfrak{s} \prec \mathfrak{Q}\left(y^{1/2}\right)$. Hence, applying~\eqref{basecase} with $y'$ in place of $y$ and $\mathfrak{s}$ in place of $\mathfrak{p}$, we obtain
\begin{equation} \label{secondterm}
\begin{split}
\Phi_{\mathfrak{rs}}(W,\mathfrak{s})&= \frac{y'}{\log y'} + O\left(\frac{y'}{\log^2 y'
}+\frac{\mathcal{N}(\mathfrak{s})}{\log \mathcal{N}(\mathfrak{s})}\right)\\ 
&= \frac{y/\mathcal{N}(\mathfrak{s})}{\log (y/\mathcal{N}(\mathfrak{s}))} + O\left(\frac{y/\mathcal{N}(\mathfrak{s})}{\log^2(y/\mathcal{N}(\mathfrak{s}))}+\frac{\mathcal{N}(\mathfrak{s})}{\log \mathcal{N}(\mathfrak{s})}\right)
\end{split}
\end{equation}
under this condition. Plugging~\eqref{firstterm} and~\eqref{secondterm} into 
\eqref{BSapp} gives
\begin{equation*} 
\begin{split}
\Phi_{\mathfrak{r}}(W,\mathfrak{p})&= \frac{y}{\log y}+\sum\limits_{\substack{\mathfrak{p}\preceq \mathfrak{s} \prec \mathfrak{Q}\left(y^{1/2}\right)}} \frac{y/\mathcal{N}(\mathfrak{s})}{\log (y/\mathcal{N}(\mathfrak{s}))}+O\left(\frac{y}{\log^2 y}\right)+\\
&\quad + O\left(\sum\limits_{\substack{\mathfrak{p}\preceq \mathfrak{s} \prec \mathfrak{Q}\left(y^{1/2}\right)}} \left(\frac{y/\mathcal{N}(\mathfrak{s})}{\log^2(y/\mathcal{N}(\mathfrak{s}))}+\frac{\mathcal{N}(\mathfrak{s})}{\log \mathcal{N}(\mathfrak{s})}\right)\right).
\end{split}
\end{equation*}
Using Lemma~\ref{LPIT}, it is easily seen that the second $O$-term can be absorbed into the first one. Moreover, a standard application of partial summation together with Lemma~\ref{LPIT} gives (cf.~\cite[page 399]{Tenenbaum}) 
\begin{equation} \label{derivation}
\sum\limits_{\substack{\mathfrak{p}\preceq \mathfrak{s} \prec \mathfrak{Q}\left(y^{1/2}\right)}} \frac{y/\mathcal{N}(\mathfrak{s})}{\log (y/\mathcal{N}(\mathfrak{s}))}=\frac{\log(u-1)}{u}\cdot \frac{y}{\log \mathcal{N}(\mathfrak{p})}+O\left(\frac{y}{\log^2 y}\right). 
\end{equation}
Altogether, we thus obtain
\begin{equation} \label{2case}
\begin{split}
\Phi_{\mathfrak{r}}(W,\mathfrak{p})&= \frac{y}{\log y}+\frac{\log(u-1)}{u}\cdot \frac{y}{\log \mathcal{N}(\mathfrak{p})}+O\left(\frac{y}{\log^2 y}\right)\\
&= \frac{1+\log(u-1)}{u}\cdot \frac{y}{\log \mathcal{N}(\mathfrak{p})}
+O\left(\frac{y}{\log^2 y}\right)\\
&= \mathcal{B}(u) \cdot \frac{y}{\log \mathcal{N}(\mathfrak{p})}+O\left(\frac{y}{\log^2 y}\right),
\end{split}
\end{equation}
which implies~\eqref{genasymp} for $2<u\le 3$. 

Next, we assume that $3<u\le 4$ and hence $y^{1/4}\le  \mathcal{N}(\mathfrak{p})< y^{1/3}$. Applying Lemma~\ref{Buchstab} with
$\mathfrak{q}=\mathfrak{Q}\left(y^{1/3}\right)$ then gives
\begin{equation} \label{BSapp2}
\Phi_{\mathfrak{r}}(W,\mathfrak{p})= \Phi_{\mathfrak{r}}\left(W,\mathfrak{Q}\left(y^{1/3}\right)\right)+\sum\limits_{\substack{\mathfrak{p}\preceq \mathfrak{s} \prec \mathfrak{Q}\left(y^{1/3}\right)}} 
\Phi_{\mathfrak{rs}}(W,\mathfrak{s}).
\end{equation}
From~\eqref{2case}, we deduce that
\begin{equation} \label{firstterm2}
\Phi_{\mathfrak{r}}\left(W,\mathfrak{Q}\left(y^{1/3}\right)\right)
=\mathcal{B}(3)\cdot \frac{y}{\log y^{1/3}}+O\left(\frac{y}{\log^2 y}\right).
\end{equation}
Further, we define $y'$ and $u'$ as in~\eqref{yu} and observe that
$$
2<u'\le 3
$$
if $\mathfrak{Q}\left(y^{1/4}\right) \preceq \mathfrak{p}\preceq \mathfrak{s} \prec \mathfrak{Q}\left(y^{1/3}\right)$. Hence, applying~\eqref{2case} with $y'$ in place of $y$ and $\mathfrak{s}$ in place of $\mathfrak{p}$, we obtain 
\begin{equation} \label{secondterm2}
\begin{split}
\Phi_{\mathfrak{rs}}(W,\mathfrak{s})& = \mathcal{B}(u')\cdot \frac{y'}{\log \mathcal{N}(\mathfrak{s})} + O\left(\frac{y'}{\log^2 y'}\right)\\ 
& = \mathcal{B}\left(\frac{\log y}{\log \mathcal{N}(\mathfrak{s})}-1\right) \cdot \frac{y/\mathcal{N}(\mathfrak{s})}{\log \mathcal{N}(\mathfrak{s})} + O\left(\frac{y/\mathcal{N}(\mathfrak{s})}{\log^2(y/\mathcal{N}(\mathfrak{s}))}\right)
\end{split}
\end{equation}
under this condition. Plugging~\eqref{firstterm2} and~\eqref{secondterm2} into 
\eqref{BSapp2} gives
\begin{equation} \label{BSapp3}
\begin{split}
\Phi_{\mathfrak{r}}(W,\mathfrak{p})&= \mathcal{B}(3)\cdot \frac{y}{\log y^{1/3}}
 +\sum\limits_{\substack{\mathfrak{p}\preceq \mathfrak{s} \prec \mathfrak{Q}\left(y^{1/3}\right)}} \mathcal{B}\left(\frac{\log y}{\log \mathcal{N}(\mathfrak{s})}-1\right) \cdot \frac{y/\mathcal{N}(\mathfrak{s})}{\log \mathcal{N}(\mathfrak{s})}+\\ & \quad + O\left(\frac{y}{\log^2 y}\right)+
 O\left(\sum\limits_{\substack{\mathfrak{p}\preceq \mathfrak{s} \prec \mathfrak{Q}\left(y^{1/3}\right)}} \frac{y/\mathcal{N}(\mathfrak{s})}{\log^2(y/\mathcal{N}(\mathfrak{s}))}\right).
\end{split}
\end{equation}
Using Lemma~\ref{LPIT}, it is again easily seen that the second $O$-term can be absorbed into the first one. Moreover, a standard application of partial summation together with Lemma~\ref{LPIT} gives (see~\cite[page 399]{Tenenbaum} for details)
\begin{equation}\label{derivation2}
\begin{aligned}
	\MoveEqLeft
\mathcal{B}(3)\cdot \frac{y}{\log y^{1/3}}
 +\sum\limits_{\substack{\mathfrak{p}\preceq \mathfrak{s} \prec \mathfrak{Q}\left(y^{1/3}\right)}} \mathcal{B}\left(\frac{\log y}{\log \mathcal{N}(\mathfrak{s})}-1\right) \cdot \frac{y/\mathcal{N}(\mathfrak{s})}{\log \mathcal{N}(\mathfrak{s})} \\ &
 = \frac{1}{u}\cdot \left(1+\int\limits_1^{u-1} \mathcal{B}(v)dv\right)\cdot \frac{y}{\log \mathcal{N}(\mathfrak{p})}+O\left(\frac{y}{\log^2 y}\right)\\ &
 =
 \mathcal{B}(u)\cdot  \frac{y}{\log \mathcal{N}(\mathfrak{p})}+O\left(\frac{y}{\log^2 y}\right).
\end{aligned}
\end{equation}
Altogether, we thus obtain
\begin{equation}
\begin{split}
\Phi_{\mathfrak{r}}(W,\mathfrak{p})= \mathcal{B}(u)\cdot \frac{y}{\log \mathcal{N}(\mathfrak{p})}+O\left(\frac{y}{\log^2 y}\right),
\end{split}
\end{equation}
which implies~\eqref{genasymp} if $3<u\le 4$. Iterating this procedure, we establish~\eqref{genasymp} for general $u$. The dependence of the $O$-term on $\beta$ therein comes from the iterations. If $k<\beta\le k+1$ with $k\in \mathbb{N}$, then we need $k$ iterations to establish~\eqref{genasymp}.  
\end{proof}

\subsection{General weight functions}
Now we extend Proposition~\ref{roughideals} to more general weight functions. Here we need to replace the prime number theorem by a suitable condition on averages of $W(\mathfrak{rs})$, where $\mathfrak{r}\in \mathcal{O}\setminus \{0\}$ is fixed and $\mathfrak{s}$ runs over prime ideals. We shall
assume that $W=W_{N} : \mathcal{I}\rightarrow \mathbb{R}^+$ is a weight function depending on a variable $N\ge 10$ which satisfies asymptotics of the form 
\begin{equation} \label{keycond}
\boxed{
\sum\limits_{\mathfrak{s}\in \mathbb{P}} W(\mathfrak{rs})=
\frac{N/\mathcal{N}(\mathfrak{r})}{\log(N/\mathcal{N}(\mathfrak{r}))}\cdot \left(1+O\left(\frac{\log\log N}{\log(N/\mathcal{N}(\mathfrak{r}))}\right)\right)
\mbox{ if } \mathcal{N}(\mathfrak{r})\le N/2.}
\end{equation}
We also assume that partial sums of the above series are bounded by
\begin{equation} \label{finitecond}
\boxed{
\sum\limits_{\substack{\mathfrak{s}\in \mathbb{P}\\ \mathfrak{s}\prec \mathfrak{p}}} W(\mathfrak{rs})=O_{\varepsilon}\left(\frac{\mathcal{N}(\mathfrak{p})}{\log \mathcal{N}(\mathfrak{p})}\right)
\mbox{ if } \mathcal{N}(\mathfrak{p})\ge N^{\varepsilon} \mbox{ and } \mathcal{N}(\mathfrak{r})\le N/2}
\end{equation} 
for any $\varepsilon>0$. 
Moreover, we assume that ``tails''
$$
\sum\limits_{\substack{\mathfrak{a}\in \mathcal{O}\\ \mathcal{N}(\mathfrak{ar})>\tilde{N}}} W(\mathfrak{ar})
$$
with $\tilde{N}$ slightly larger than $N$ are small. Our precise condition is
\begin{equation} \label{tailcond}
\boxed{
\sum\limits_{\substack{\mathfrak{a}\in \mathcal{O}\\ \mathcal{N}(\mathfrak{ar})>\tilde{N}}} W(\mathfrak{ar})=O_{\xi}\left(\frac{N/\mathcal{N}(\mathfrak{r})}{\log^2 (N/\mathcal{N}(\mathfrak{r}))}\right) \mbox{ if } \tilde{N}=N\log^{\xi} N \mbox{ and }
\mathcal{N}(\mathfrak{r})\le N/2}
\end{equation}
for some $\xi\in (0,1)$. Generalizing Proposition~\ref{roughideals}, we establish the following. 

\begin{Theorem} \label{genweights} Fix $\varepsilon,\xi\in (0,1)$. Assume that
$W=W_{N}:\mathcal{I}\Rightarrow \mathbb{R}_{\ge 0}$ is a function depending on a variable $N\ge 10$ which 
satisfies the conditions~\eqref{keycond}, \eqref{finitecond} and~\eqref{tailcond}.
Let $\beta>2>\alpha>1$, $\mathfrak{p}\in \mathbb{P}$ and $\mathfrak{r}\in \mathcal{O}\setminus \{0\}$ and set
$$
y\coloneqq\frac{N}{\mathcal{N}(\mathfrak{r})} \quad \mbox{and} \quad u\coloneqq\frac{\log y}{\log \mathcal{N}(\mathfrak{p})}.
$$
Assume that
$$
\mathcal{N}(\mathfrak{r})\le N^{1-\varepsilon} \quad \mbox{and} \quad \alpha\le u \le \beta.
$$
Then,
\begin{equation} \label{genasymp3}
\Phi_{\mathfrak{r}}\left(W,\mathfrak{p}\right)= \mathcal{B}(u)
\cdot \frac{y}{\log \mathcal{N}(\mathfrak{p})}\cdot \left(1+O_{\alpha,\beta,\varepsilon,\xi}\left(\log^{\xi-1} y\right)\right),
\end{equation}
where $\mathcal{B}(u)$ is the Buchstab function. 
\end{Theorem}

\begin{proof} We shall imitate the proof of Proposition~\ref{roughideals}. In this proof, all implied $O$-constants will be allowed to depend on $\alpha,\beta,\varepsilon,\xi$.
We define two more parameters $\tilde{y}$ and $\tilde{u}$ by
$$
\tilde{y}\coloneqq\frac{\tilde{N}}{\mathcal{N}(\mathfrak{r})}=\frac{N\log^{\xi} N}{\mathcal{N}(\mathfrak{r})}\quad \mbox{and} \quad \tilde{u}\coloneqq\frac{\log \tilde{y}}{\log \mathcal{N}(\mathfrak{p})}
$$
and establish that 
\begin{equation} \label{genasymp4}
\Phi_{\mathfrak{r}}\left(W,\mathfrak{p}\right)= \mathcal{B}(\tilde{u})
\cdot \frac{y}{\log \mathcal{N}(\mathfrak{p})}\cdot \left(1+O\left(\log^{-\xi} y \right)\right).
\end{equation}
We observe that
\begin{equation} \label{uu'}
|\mathcal{B}(\tilde{u})-\mathcal{B}(u)|\ll_{\beta} |\tilde{u}-u| = \frac{\log(\log N)^{\xi}}{\log  \mathcal{N}(\mathfrak{p})}\ll_{\beta,\varepsilon,\xi} \frac{\log\log N}{\log y}
\ll_{\beta,\varepsilon,\xi} \frac{1}{\log^{1-\xi}y},  
\end{equation}
where we take into account that $y\ge N^{\varepsilon}$ by the assumptions in Theorem~\ref{genweights}.
Hence, \eqref{genasymp4} implies~\eqref{genasymp3}. In the following we will stop indicating possible dependencies of the implied constants on $\alpha,\beta,\varepsilon,\xi$. 

Similarly as in the proof of Proposition~\eqref{roughideals}, we begin with considering the range $1<\tilde{u}\le 2$ which corresponds to $\tilde{y}^{1/2}\le \mathcal{N}(\mathfrak{p})< \tilde{y}$. We write 
\begin{equation} \label{basecase'}
\begin{split}
\Phi_{\mathfrak{r}}(W,\mathfrak{p})= \sum_{\substack{\mathcal{N}(\mathfrak{a})\le \tilde{y}\\(\mathfrak{a},\Pi(\mathfrak{p}))=1}} W(\mathfrak{ar}) +
\sum_{\substack{\mathcal{N}(\mathfrak{a})> \tilde{y}\\(\mathfrak{a},\Pi(\mathfrak{p}))=1}} W(\mathfrak{ar}).
\end{split}
\end{equation}
Using condition~\eqref{tailcond}, the second sum on the right-hand side is bounded by 
\begin{equation} \label{tailsum}
\sum_{\substack{\mathcal{N}(\mathfrak{a})> \tilde{y}\\(\mathfrak{a},\Pi(\mathfrak{p}))=1}} W(\mathfrak{ar})=O\left(\frac{y}{\log^2 y}\right).
\end{equation}
The first sum on the right-hand side of~\eqref{basecase'} equals 
\begin{equation*}
\sum_{\substack{\mathcal{N}(\mathfrak{a})\le \tilde{y}\\(\mathfrak{a},\Pi(\mathfrak{p}))=1}} W(\mathfrak{ar})=\sum\limits_{\substack{\mathfrak{s}\in \mathbb{P}\\ \mathfrak{p}\preceq \mathfrak{s} \prec \mathfrak{Q}(\tilde{y})}} W(\mathfrak{sr}) 
\end{equation*}
since $\tilde{y}^{1/2}\le \mathcal{N}(\mathfrak{p})< \tilde{y}$. Using all three conditions~\eqref{keycond}, \eqref{finitecond} and~\eqref{tailcond}, we deduce that
\begin{equation} \label{mainsum}
\sum_{\substack{\mathcal{N}(\mathfrak{a})\le \tilde{y}\\(\mathfrak{a},\Pi(\mathfrak{p}))=1}} W(\mathfrak{ar})
=\frac{y}{\log y} + O\left(\frac{y}{\log^{2-\xi} y}+\frac{\mathcal{N}(\mathfrak{p})}{\log \mathcal{N}(\mathfrak{p})}\right).
\end{equation}
A short calculation confirms that
\begin{equation} \label{short}
\frac{y}{\log y}
=\frac{y}{\log \tilde{y}}+O\left(\frac{y}{\log^{2-\xi} y}\right).
\end{equation}
Combining~\eqref{basecase'}, \eqref{tailsum}, \eqref{mainsum} and~\eqref{short}, we obtain
\begin{equation} \label{basecase''}
\Phi_{\mathfrak{r}}(W,\mathfrak{p})=\frac{y}{\log \tilde{y}} + O\left(\frac{y}{\log^{2-\xi} y}+\frac{\mathcal{N}(\mathfrak{p})}{\log \mathcal{N}(\mathfrak{p})}\right)
\end{equation}
if $1<\tilde{u}\le 2$. If, in addition, $u\ge \alpha>1$, then making use of our bound for $|\tilde{u}-u|$ in~\eqref{uu'}, we see that the second $O$-term is majorized by the first term and hence,
\begin{equation*}
\Phi_{\mathfrak{r}}(W,\mathfrak{p})=\frac{y}{\log \tilde{y}} + O\left(\frac{y}{\log^{2-\xi} y}\right). 
\end{equation*}
Now~\eqref{genasymp4} follows under the above conditions $1<\tilde{u}\le 2$ and $u\ge  \alpha$. 

Next we turn to the range $2<\tilde{u}\le 3$ which corresponds to $\tilde{y}^{1/3}\le  \mathcal{N}(\mathfrak{p})< \tilde{y}^{1/2}$. In this case, Lemma~\ref{Buchstab} with
$\mathfrak{q}=\mathfrak{Q}\left(\tilde{y}^{1/2}\right)$ gives
\begin{equation} \label{BSapp'}
\Phi_{\mathfrak{r}}(W,\mathfrak{p})= \Phi_{\mathfrak{r}}\left(W,\mathfrak{Q}\left(\tilde{y}^{1/2}\right)\right)+\sum\limits_{\substack{\mathfrak{p}\preceq \mathfrak{s} \prec \mathfrak{Q}\left(\tilde{y}^{1/2}\right)}} 
\Phi_{rs}(W,\mathfrak{s}).
\end{equation}
Using~\eqref{basecase''}, we have
\begin{equation} \label{firstterm'}
\Phi_{\mathfrak{r}}\left(W,\mathfrak{Q}\left(\tilde{y}^{1/2}\right)\right)
=\frac{y}{\log \tilde{y}}+O\left(\frac{y}{\log^{2-\xi} y}\right).
\end{equation}
Further, we set  
\begin{equation} \label{yu'}
y'\coloneqq\frac{N}{\mathcal{N}(\mathfrak{rs})}=\frac{y}{\mathcal{N}(\mathfrak{s})}, \quad \tilde{y}'\coloneqq\frac{\tilde{N}}{\mathcal{N}(\mathfrak{rs})}=\frac{\tilde{y}}{\mathcal{N}(\mathfrak{s})}\quad \mbox{and} \quad 
u'\coloneqq\frac{\log \tilde{y}'}{\log \mathcal{N}(\mathfrak{s})}=\frac{\log \tilde{y}}{\log \mathcal{N}(\mathfrak{s})}-1
\end{equation}
and observe that
$$
1<u'\le 2
$$
if $\mathfrak{Q}\left(\tilde{y}^{1/3}\right) \preceq \mathfrak{p}\preceq \mathfrak{s} \prec \mathfrak{Q}\left(\tilde{y}^{1/2}\right)$. Hence, applying~\eqref{basecase''} with $y'$ in place of $y$, $\tilde{y}'$ in place of $\tilde{y}$ and $\mathfrak{s}$ in place of $\mathfrak{p}$, we obtain
\begin{equation} \label{secondterm'}
\begin{split}
\Phi_{\mathfrak{rs}}(W,\mathfrak{s})& = \frac{y'}{\log \tilde{y}'} + O\left(\frac{y'}{\log^{2-\xi} y'
}+\frac{\cdot \mathcal{N}(\mathfrak{s})}{\log \mathcal{N}(\mathfrak{s})}\right)\\ 
& = \frac{y/\mathcal{N}(\mathfrak{s})}{\log (\tilde{y}/\mathcal{N}(\mathfrak{s}))} + O\left(\frac{y/\mathcal{N}(\mathfrak{s})}{\log^{2-\xi}(y/\mathcal{N}(\mathfrak{s}))}+\frac{\mathcal{N}(\mathfrak{s})}{\log \mathcal{N}(\mathfrak{s})}\right)
\end{split}
\end{equation}
under this condition. Plugging~\eqref{firstterm'} and~\eqref{secondterm'} into 
\eqref{BSapp'} gives
\begin{equation*}
\begin{split}
\Phi_{\mathfrak{r}}(W,\mathfrak{p})& = \frac{y}{\log \tilde{y}}+\sum\limits_{\substack{\mathfrak{p}\preceq \mathfrak{s} \prec \mathfrak{Q}\left(\tilde{y}^{1/2}\right)}} \frac{y/\mathcal{N}(\mathfrak{s})}{\log (\tilde{y}/\mathcal{N}(\mathfrak{s}))}+O\left(\frac{y}{\log^{2-\xi} y}\right)+\\
& \quad +O\left(\sum\limits_{\substack{\mathfrak{p}\preceq \mathfrak{s} \prec \mathfrak{Q}\left(\tilde{y}^{1/2}\right)}} \left(\frac{y/\mathcal{N}(\mathfrak{s})}{\log^{2-\xi}(y/\mathcal{N}(\mathfrak{s}))}+\frac{\mathcal{N}(\mathfrak{s})}{\log \mathcal{N}(\mathfrak{s})}\right)\right).
\end{split}
\end{equation*}
Using Lemma~\ref{LPIT}, the second $O$-term is of size
\begin{align*}
	\MoveEqLeft
	\sum\limits_{\substack{\mathfrak{p}\preceq \mathfrak{s} \prec \mathfrak{Q}\left(\tilde{y}^{1/2}\right)}} \left(\frac{y/\mathcal{N}(\mathfrak{s})}{\log^{2-\xi}(y/\mathcal{N}(\mathfrak{s}))}+\frac{\mathcal{N}(\mathfrak{s})}{\log \mathcal{N}(\mathfrak{s})}\right)\\ &\ll \frac{y}{\log^{2-\xi} y}+\frac{\tilde{y}}{\log^2 y} =\frac{y}{\log^{2-\xi} y}+\frac{y\log^{\xi} N}{\log^2 y}
\ll \frac{y}{\log^{2-\xi} y}.
\end{align*}
A standard application of partial summation together with Lemma~\ref{LPIT} like in the derivation of~\eqref{derivation} gives 
$$
\sum\limits_{\substack{\mathfrak{p}\preceq \mathfrak{s} \prec \mathfrak{Q}\left(\tilde{y}^{1/2}\right)}} \frac{y/\mathcal{N}(\mathfrak{s})}{\log (\tilde{y}/\mathcal{N}(\mathfrak{s}))}=\frac{\log(\tilde{u}-1)}{\tilde{u}}\cdot \frac{y}{\log \mathcal{N}(\mathfrak{p})}+O\left(\frac{y}{\log^2 y}\right). 
$$
Altogether, we thus obtain
\begin{equation} \label{2case'}
\begin{split}
\Phi_{\mathfrak{r}}(W,\mathfrak{p})&= \frac{y}{\log \tilde{y}}+\frac{\log(\tilde{u}-1)}{\tilde{u}}\cdot \frac{y}{\log \mathcal{N}(\mathfrak{p})}+O\left(\frac{y}{\log^{2-\xi} y}\right)\\
&= \frac{1+\log(\tilde{u}-1)}{\tilde{u}}\cdot \frac{y}{\log \mathcal{N}(\mathfrak{p})}
+O\left(\frac{y}{\log^{2-\xi} y}\right)\\
&= \mathcal{B}(\tilde{u}) \cdot \frac{y}{\log \mathcal{N}(\mathfrak{p})}+O\left(\frac{y}{\log^{2-\xi} y}\right),
\end{split}
\end{equation}
which confirms~\eqref{genasymp4} for $2<\tilde{u}\le 3$. 

Next, we assume that $3<\tilde{u}\le 4$ and hence $\tilde{y}^{1/4}\le  \mathcal{N}(\mathfrak{p})< \tilde{y}^{1/3}$. Applying Lemma~\ref{Buchstab} with
$\mathfrak{q}=\mathfrak{Q}\left(\tilde{y}^{1/3}\right)$ then gives
\begin{equation} \label{BSapp2'}
\Phi_{\mathfrak{r}}(W,\mathfrak{p})= \Phi_{\mathfrak{r}}\left(W,\mathfrak{Q}\left(\tilde{y}^{1/3}\right)\right)+\sum\limits_{\substack{\mathfrak{p}\preceq \mathfrak{s} \prec \mathfrak{Q}\left(\tilde{y}^{1/3}\right)}} 
\Phi_{\mathfrak{rs}}(W,\mathfrak{s}).
\end{equation}
From~\eqref{2case'}, we deduce that
\begin{equation} \label{firstterm2'}
\Phi_{\mathfrak{r}}\left(W,\mathfrak{Q}\left(\tilde{y}^{1/3}\right)\right)
=\mathcal{B}(3)\cdot \frac{y}{\log \tilde{y}^{1/3}}+O\left(\frac{y}{\log^{2-\xi} y}\right).
\end{equation}
Further, we define $y'$, $\tilde{y}'$ and $u'$ as in~\eqref{yu'} and observe that
$$
2<u'\le 3
$$
if $\mathfrak{Q}\left(\tilde{y}^{1/4}\right) \preceq \mathfrak{p}\preceq \mathfrak{s} \prec \mathfrak{Q}\left(\tilde{y}^{1/3}\right)$. Hence, applying~\eqref{2case'} with $y'$ in place of $y$, $\tilde{y}'$ in place of $\tilde{y}$, $u'$ in place of $\tilde{u}$ and $\mathfrak{s}$ in place of $\mathfrak{p}$, we obtain 
\begin{equation} \label{secondterm2'}
\begin{split}
\Phi_{\mathfrak{rs}}(W,\mathfrak{s})&= \mathcal{B}(u')\cdot \frac{y'}{\log \mathcal{N}(\mathfrak{s})} + O\left(\frac{y'}{(\log y')^{2-\xi}}\right)\\ 
&= \mathcal{B}\left(\frac{\log \tilde{y}}{\log \mathcal{N}(\mathfrak{s})}-1\right) \cdot \frac{y/\mathcal{N}(\mathfrak{s})}{\log \mathcal{N}(\mathfrak{s})} + O\left(\frac{y/\mathcal{N}(\mathfrak{s})}{\log^{2-\xi}(y/\mathcal{N}(\mathfrak{s}))}\right)
\end{split}
\end{equation}
under this condition. Plugging~\eqref{firstterm2'} and~\eqref{secondterm2'} into 
\eqref{BSapp2'} gives
\begin{equation} \label{BSapp3'}
\begin{split}
\Phi_{\mathfrak{r}}(W,\mathfrak{p})&= \mathcal{B}(3)\cdot \frac{y}{\log \tilde{y}^{1/3}}
 +\sum\limits_{\substack{\mathfrak{p}\preceq \mathfrak{s} \prec \mathfrak{Q}\left(\tilde{y}^{1/3}\right)}} \mathcal{B}\left(\frac{\log \tilde{y}}{\log \mathcal{N}(\mathfrak{s})}-1\right) \cdot \frac{y/\mathcal{N}(\mathfrak{s})}{\log \mathcal{N}(\mathfrak{s})}+\\ & \quad + O\left(\frac{y}{\log^{2-\xi} y}\right)+
 O\left(\sum\limits_{\substack{\mathfrak{p}\preceq \mathfrak{s} \prec \mathfrak{Q}\left(\tilde{y}^{1/3}\right)}} \frac{y/\mathcal{N}(\mathfrak{s})}{\log^{2-\xi}(y/\mathcal{N}(\mathfrak{s}))}\right).
\end{split}
\end{equation}
Using Lemma~\ref{LPIT}, it is again easily seen that the second $O$-term can be absorbed into the first one. Moreover, a standard application of partial summation together with Lemma~\ref{LPIT} like in the derivation of~\eqref{derivation2} gives 
\begin{align*}
	\MoveEqLeft
	\mathcal{B}(3)\cdot \frac{y}{\log \tilde{y}^{1/3}}
 +\sum\limits_{\substack{\mathfrak{p}\preceq \mathfrak{s} \prec \mathfrak{Q}\left(\tilde{y}^{1/3}\right)}} \mathcal{B}\left(\frac{\log \tilde{y}}{\log \mathcal{N}(\mathfrak{s})}-1\right) \cdot \frac{y/\mathcal{N}(\mathfrak{s})}{\log \mathcal{N}(\mathfrak{s})}\\ & =  \frac{1}{\tilde{u}}\cdot \left(1+\int\limits_1^{\tilde{u}-1} \mathcal{B}(v)dv\right)\cdot \frac{y}{\log \mathcal{N}(\mathfrak{p})}+O\left(\frac{y}{\log^2 y}\right)\\
 &= 
 \mathcal{B}(\tilde{u})\cdot  \frac{y}{\log \mathcal{N}(\mathfrak{p})}+O\left(\frac{y}{\log^2 y}\right).
\end{align*}
Altogether, we thus obtain
\begin{equation}
\begin{split}
\Phi_{\mathfrak{r}}(W,\mathfrak{p})= \mathcal{B}(\tilde{u})\cdot \frac{y}{\log \mathcal{N}(\mathfrak{p})}+O\left(\frac{y}{\log^{2-\xi} y}\right)
\end{split}
\end{equation}
if $3<\tilde{u}\le 4$. Similarly as in the proof Proposition~\ref{roughideals}, we can now iterate the above procedure to establish~\eqref{genasymp4} for general $\tilde{u}$. This completes the proof.  
\end{proof}
 
\section{Harman's lower bound sieve for number fields}
In this section, we prove the following version of Harman's lower bound sieve for number fields, generalizing the lower bound part of~\cite[Theorem 2]{harman1996on-the-distribu}.

\begin{Theorem}[Harman's lower bound sieve for $\mathcal{I}$] \label{lower} 
Fix $\theta\in [1/4,1/3]$ and $\varepsilon,\xi\in (0,1)$. Let $\lambda>0$ and $N\ge 3$ be two variables. Assume that
$W=W_{N}:\mathcal{I}\rightarrow \mathbb{R}_{\ge 0}$ is a function depending on $N$ and  $\tilde{\omega}=\tilde{\omega}_{\lambda,N}:\mathcal{I}\rightarrow \mathbb{R}_{\ge 0}$ is a function depending on $\lambda$ and $N$.  Suppose that  $W$ satisfies the    
conditions~\eqref{keycond}, \eqref{finitecond} and~\eqref{tailcond}. Set $\omega=\lambda W$ and let $x\ge 3$ be a real number such that
$
N\sim x^{1-\varepsilon}.
$ 
Suppose that for $w=\omega,\tilde{\omega}$, we have
\begin{equation} \label{Xcond}
\boxed{
  \sum_{\mathfrak{n}\in \mathcal{I}} d_5(\mathfrak{n})w(\mathfrak{n})\leq x^A}
\end{equation}
for some $A>0$ and 
\begin{equation} \label{newsmallandlarge}
\boxed{
\sum\limits_{\substack{\mathfrak{n}\in \mathcal{I}\\ \mathcal{N}(\mathfrak{n})\not\in (x^{1-\iota},x)}} d_5(\mathfrak{n})w(\mathfrak{n}) \leq \lambda N^{1-\eta}}
\end{equation}
for some fixed $\iota\in [\theta-1/4,\theta]$ and $\eta>0$. 
Suppose further that for any sequences $(a_\mathfrak{a})_{\mathfrak{a} \in \mathcal{I}}$,$(b_\mathfrak{b})_{\mathfrak{b} \in \mathcal{I}}$ of complex numbers with 
$|a_\mathfrak{a}|\leq d_3(\mathfrak{a})$ and $|b_\mathfrak{b}|\leq d_3(\mathfrak{b})$, the inequalities 
\begin{equation} \label{typeI} 
    |\mathop{\sum\sum}\limits_{\substack{\mathfrak{a},\mathfrak{b} \in \mathcal{I}\\ \mathcal{N}(\mathfrak{a})\le x^{1-\theta}}} a_\mathfrak{a} (\omega(\mathfrak{ab})-\tilde{\omega}(\mathfrak{ab}))| \leq \lambda N^{1-\eta}
\end{equation}
and
\begin{equation} \label{typeII}
    |\mathop{\sum\sum}\limits_{\substack{\mathfrak{a},\mathfrak{b} \in \mathcal{I}\\ 
    x^{\theta-\iota}\le \mathcal{N}(\mathfrak{a})\le x^{1-2\theta}}} a_\mathfrak{a} b_\mathfrak{b} (\omega(\mathfrak{ab})-\tilde{\omega}(\mathfrak{ab}))| \leq 
\lambda N^{1-\eta}
\end{equation}
are satisfied.  
Then we have 
\begin{equation} \label{secondclaim}
   \frac{S_{\mathcal{O}}(\tilde{\omega},\sqrt{x})}{S_{\mathcal{O}}(\omega,\sqrt{x})}\ge C(\theta)+o(1)
\end{equation}
as $N\rightarrow \infty$.  
Here $C(\theta)$ is a monotonic and continuous function satisfying $C(\theta)=1+O((\theta-1/4)^{2})$ and $C(7/22)>1/10$.
\end{Theorem}

\begin{proof}
We shall literally translate~\cite[section 4]{harman1996on-the-distribu} into our more general setting. 
A lot of notations will be kept. We only arrange the material in a slightly different order and make minor adjustments. The rough idea of the proof is as follows. As in the proof of Theorem~\ref{asymp}, we first note that  
$$
S_{\mathcal{O}}(\tilde{\omega},\sqrt{x})=\Phi_{\mathcal{O}}(\tilde{\omega},\mathfrak{q}_0),
$$
where $\mathfrak{q}_0\coloneqq\mathfrak{Q}\left(\sqrt{x}\right)$. We 
decompose $\Phi_{\mathcal{O}}(\tilde{\omega},\mathfrak{q}_0)$ into a number of sums by iteration of 
Buchstab's identity, Lemma~\ref{Buchstab}. A part of the resulting terms can be discarded immediately if they have positive sign since we are only interested in lower bounds. Another part is approximated by corresponding terms with $\omega$ in place of $\tilde{\omega}$, using Theorem~\ref{asymp} in conjunction with~\eqref{typeI} and~\eqref{typeII} or by a direct application of~\eqref{typeII}.  We are left with a sum containing just terms of the form  $\Phi_{r}(\omega,\mathfrak{p})$. Now we reverse all applications of Buchstab's identity which allows us to greatly simplify the said sum. We approximate the resulting shorter sum of terms of the form $\Phi_{r}(\omega,\mathfrak{p})$ using Theorem~\ref{genweights}. An application of Landau's  prime ideal theorem together with partial summation then leads to integrals which were calculated in ~\cite[section 4]{harman1996on-the-distribu}. We shall carry out in detail only the proof for $\theta\le 2/7$. In the case $\theta>2/7$, which requires additional decompositions, we cut the calculations short and refer to the parallel treatment in~\cite[section 4]{harman1996on-the-distribu}.

Let 
$$
\mathfrak{q}_0\coloneqq\mathfrak{Q}(x^{1/2}) \quad \mbox{and} \quad \mathfrak{q}_1\coloneqq
\mathfrak{Q}(x^{\kappa}),\quad \mbox{where } \kappa\coloneqq1-3\theta.
$$
Then applying Lemma~\ref{Buchstab} twice yields
\begin{equation} \label{Buchtwice}
\begin{split}
\Phi_{\mathcal{O}}(\tilde{\omega},\mathfrak{q}_0) &= \Phi_{\mathcal{O}}(\tilde{\omega},\mathfrak{q}_1)-
\sum\limits_{\mathfrak{q}_1\preceq \mathfrak{p}\prec \mathfrak{q}_0} 
\Phi_{\mathfrak{p}}(\tilde{\omega},\mathfrak{p})\\
 &= \Phi_{\mathcal{O}}(\tilde{\omega},\mathfrak{q}_1)- \sum\limits_{\substack{\mathfrak{q}_1\preceq \mathfrak{p}\prec \mathfrak{Q}(x^{\theta}) \\ \mbox{\scriptsize or } \mathfrak{Q}(x^{1-2\theta})\preceq \mathfrak{p}\prec \mathfrak{q}_0}} 
\Phi_{\mathfrak{p}}(\tilde{\omega},\mathfrak{p})-\sum\limits_{\mathfrak{Q}(x^{\theta})\preceq \mathfrak{p} \prec \mathfrak{Q}(x^{1-2\theta})} \Phi_{\mathfrak{p}}(\tilde{\omega},\mathfrak{p})\\
 &=  \Phi_{\mathcal{O}}(\tilde{\omega},\mathfrak{q}_1)-
\sum\limits_{\substack{\mathfrak{q}_1\preceq \mathfrak{p}\prec \mathfrak{Q}(x^{\theta}) \\ \mbox{\scriptsize or } \mathfrak{Q}(x^{1-2\theta})\preceq \mathfrak{p}\prec \mathfrak{q}_0}} 
\Phi_{\mathfrak{p}}(\tilde{\omega},\mathfrak{q}_1) - \sum\limits_{\mathfrak{Q}(x^{\theta})\preceq \mathfrak{p} \prec \mathfrak{Q}(x^{1-2\theta})} \Phi_{\mathfrak{p}}(\tilde{\omega},\mathfrak{p})+ \\
&\quad +\sum\limits_{\substack{\mathfrak{q}_1\preceq\mathfrak{q}\prec \min\left\{\mathfrak{p}, \mathfrak{Q}\left(x^{1/2}/\mathcal{N}(\mathfrak{p})^{1/2}\right)\right\}\\ \mathfrak{p}\prec \mathfrak{Q}(x^{\theta}) \ \mbox{\scriptsize or } \mathfrak{Q}(x^{1-2\theta})\preceq \mathfrak{p}\prec \mathfrak{q}_0}} \Phi_{\mathfrak{pq}}(\tilde{\omega},\mathfrak{q})+O\left(\lambda N^{1-\eta}\right)\\
 &= \tilde{S}_1-\tilde{S}_2-\tilde{S}_3+\tilde{S}_4+O\left(\lambda N^{1-\eta}\right), \quad \mbox{say,}
\end{split}
\end{equation}
where we have used~\eqref{newsmallandlarge} to truncate the sum over $\mathfrak{q}$ in $\tilde{S}_4$ at $ \mathfrak{Q}\left(x^{1/2}/\mathcal{N}(\mathfrak{p})^{1/2}\right)$. 
We split $\tilde{S}_4$ into
\begin{equation} \label{S4split}
\tilde{S}_4= \sum\limits_{(\nabla)} \Phi_{\mathfrak{pq}}(\tilde{\omega},\mathfrak{q})+\sum\limits_{(\overline{\nabla})} \Phi_{\mathfrak{pq}}(\tilde{\omega},\mathfrak{q})=\tilde{S}_5+\tilde{S}_6, \quad \mbox{say,}
\end{equation}
where 
\begin{equation*} 
(\nabla) \Longleftrightarrow \mathfrak{q}\prec \mathfrak{p}, \ \mathcal{N}(\mathfrak{p})=x^{\alpha}, \
\mathcal{N}(\mathfrak{q})=x^{\beta},\ (\alpha,\beta)\in \mathcal{M}
\end{equation*}
with 
\begin{equation*}
\begin{split}
\mathcal{M}\coloneqq \{(\alpha,\beta)\ :\ & \alpha \in [\kappa,\theta)\cup [1-2\theta,1/2),
\ \kappa\le \beta\le \alpha, \ \beta<(1-\alpha)/2, \\ & 
\{\beta,\alpha+\beta\}\cap
([\theta,1-2\theta)\cup [2\theta,1-\theta))=\emptyset\},
\end{split}
\end{equation*}
and $(\overline{\nabla})$ indicates the same condition as $(\nabla)$, except that
$$
\{\beta,\alpha+\beta\}\cap
([\theta,1-2\theta)\cup [2\theta,1-\theta))\not=\emptyset.
$$
The conditions $(\nabla)$ and $(\overline{\nabla})$ above correspond to those in~\cite[section 4]{harman1996on-the-distribu}, with the minor modifications (which don't change the final result) that our intervals therein are half-open except for the $\beta$-range $\kappa\le \beta\le \alpha$ which is $\kappa\le \beta<\alpha$ in Harman's setting. The latter modification is necessary because
it may happen that $\mathcal{N}(\mathfrak{q})=\mathcal{N}(\mathfrak{p})$ although $\mathfrak{q}\prec \mathfrak{p}$. 

Next, we split $\tilde{S}_5$ into
$$
\tilde{S}_5=\sum\limits_{\substack{(\nabla)\\ \mathcal{N}(\mathfrak{p}\mathfrak{q}^2)< x^{1-\theta}}} \Phi_{\mathfrak{pq}}(\tilde{\omega},\mathfrak{q})+
\sum\limits_{\substack{(\nabla)\\ \mathcal{N}(\mathfrak{p}\mathfrak{q}^2)\ge x^{1-\theta}}} \Phi_{\mathfrak{pq}}(\tilde{\omega},\mathfrak{q})=\tilde{S}_7+\tilde{S}_8, \quad \mbox{say.}
$$ 
Applying Buchstab's identity twice to the sum $\tilde{S}_7$ gives a decomposition of the form
\begin{equation} \label{(24)}
\tilde{S}_7= \sum \Phi_{\mathfrak{pq}}(\tilde{\omega},\mathfrak{q}_1)-\sum \Phi_{\mathfrak{pqr}}(\tilde{\omega},\mathfrak{q}_1)+\sum \Phi_{\mathfrak{pqrs}}(\tilde{\omega},\mathfrak{s})=\tilde{S}_9-\tilde{S}_{10}+\tilde{S}_{11}, \quad \mbox{say,}
\end{equation}
with the obvious summation ranges. 
We split the sum $\tilde{S}_8$, in which we have the summation condition
\begin{equation} \label{twosumconds}
(\nabla) \quad \mbox{and} \quad 
\mathcal{N}\left(\mathfrak{p}\mathfrak{q}^2\right)\ge x^{1-\theta},
\end{equation}
into 
$$
\tilde{S}_8= \sum\limits_{\substack{(\nabla\nabla)}} \Phi_{\mathfrak{pq}}(\tilde{\omega},\mathfrak{q}) + \sum\limits_{\substack{(\overline{\nabla\nabla})}} \Phi_{\mathfrak{pq}}(\tilde{\omega},\mathfrak{q})= \tilde{S}_{12}+\tilde{S}_{13}, \quad \mbox{say.}
$$ 
Here $(\nabla\nabla)$ stands for~\eqref{twosumconds} combined with the condition 
\begin{equation} \label{pqcond}
\mathcal{N}(\mathfrak{pq})< x^{1-\theta}\quad \mbox{and} \quad \mathcal{N}(\mathfrak{p})\ge x^{1-2\theta}.
\end{equation}
The summation range in the second sum $\tilde{S}_{13}$ then satisfies a condition $(\overline{\nabla\nabla})$ which is~\eqref{twosumconds} combined with  
$$
\mathcal{N}(\mathfrak{pq})\ge x^{1-\theta}\quad \mbox{or} \quad \mathcal{N}(\mathfrak{p})< x^{\theta}.
$$
The sum $\tilde{S}_{12}$ is treated by switching roles of variables as follows. We observe that  
\begin{equation} \label{reformu}
\tilde{S}_{12}=
\sum\limits_{(\Xi)} \Phi_{\mathfrak{qt}}\left(\tilde{\omega}_{\mathfrak{q,t}},\mathfrak{Q}\left(x^{1/2}/\mathcal{N}(\mathfrak{qt})^{1/2}\right)\right)
+O\left(\lambda N^{1-\eta}\right),
\end{equation}
with the summation condition $(\Xi)$ indicating that
$$
\mathcal{N}(\mathfrak{q})x^{\theta}>\mathcal{N}(\mathfrak{t})\ge x^{\theta-\iota}, \quad x^{1/2}\le \mathcal{N}(\mathfrak{qt})<x^{1-\theta}, \quad 
x^{\kappa}\le \mathcal{N}(\mathfrak{q})<x^{\theta}, \quad \mathfrak{p}|\mathfrak{t} \Rightarrow \mathfrak{q}\preceq \mathfrak{p}
$$
and $\tilde{\omega}_{\mathfrak{q,t}}$ being defined as
\begin{equation} \label{omegaq}
\tilde{\omega}_{\mathfrak{q,t}}(\mathfrak{nqt})=\tilde{\omega}(\mathfrak{nqt})1_{(\mathfrak{n},\mathfrak{q})\ \mbox{\scriptsize satisfies } (\nabla\nabla)}.
\end{equation}
Here we have again used~\eqref{newsmallandlarge} to truncate the sum over $\mathfrak{t}$. (In the above, note that $\mathfrak{n}$ takes the role of the original variable $\mathfrak{p}$, and $\mathfrak{p}$ is now a new prime ideal.) Applying Buchstab's identity to the right-hand side of~\eqref{reformu} and reversing roles of variables again, followed by another application of Buchstab's identity gives
\begin{equation} \label{(25)}
\begin{split}
\tilde{S}_{12} &= \sum\limits_{(\Xi)} \Phi_{\mathfrak{qt}}(\tilde{\omega}_{\mathfrak{q,t}},\mathfrak{q}_1)-
\sum\limits_{\substack{(\Xi)\\ \mathfrak{q}_1\preceq \mathfrak{s}\prec \mathfrak{Q}\left(x^{1/2}/\mathcal{N}(\mathfrak{qt})^{1/2}\right)}} \Phi_{\mathfrak{qts}}(\tilde{\omega}_{\mathfrak{q,t}},\mathfrak{s})+O\left(\lambda N^{1-\eta}\right)\\
 &= \tilde{S}_{14}-\sum\limits_{\substack{\mathfrak{q},\mathfrak{u}, \mathfrak{s}\\ \mathfrak{p}|\mathfrak{u} \Rightarrow \mathfrak{q}_1\preceq \mathfrak{s}\preceq \mathfrak{p}\\ 
(\mathfrak{us},\mathfrak{q}) \ \mbox{\scriptsize satisfies } (\nabla\nabla)}} \Phi_{\mathfrak{qus}}(\tilde{\omega},\mathfrak{q})+O\left(\lambda N^{1-\eta}\right), \quad \mbox{say,}\\
 &= \tilde{S}_{14}-\sum \Phi_{\mathfrak{qus}}(\tilde{\omega},\mathfrak{q}_1)+\sum \Phi_{\mathfrak{qusr}}(\tilde{\omega},\mathfrak{r})+O\left(\lambda N^{1-\eta}\right)\\
 &= \tilde{S}_{14}-\tilde{S}_{15}+\tilde{S}_{16}+O\left(\lambda N^{1-\eta}\right), \quad \mbox{say,}
\end{split}
\end{equation}
with the obvious summation ranges in the third line. Here we have again used~\eqref{newsmallandlarge}. 
For the last decomposition we take into account that $\mathcal{N}(\mathfrak{qus})< x^{1-\theta}$ follows from~\eqref{pqcond}. 

Combining everything above, we end up with a decomposition of the form
$$
\Phi(\tilde{\omega},\mathfrak{q}_0)=\tilde{S}_1-\tilde{S}_2-\tilde{S}_3+\tilde{S}_6+\tilde{S}_9-\tilde{S}_{10}+\tilde{S}_{11}+\tilde{S}_{13}+\tilde{S}_{14}-\tilde{S}_{15}+\tilde{S}_{16}+O\left(\lambda N^{1-\eta}\right).
$$
If $\theta\le 2/7$, then except for the sum $\tilde{S}_{13}$, it will turn out that all sums $\tilde{S}_i$ on the right-hand side can be approximated by the corresponding sums $S_i$ with $\omega$ in place of $\tilde{\omega}$ using~\eqref{typeII} or Theorem~\ref{asymp}. This comes at the cost of an error of size $O\left(\lambda N^{1-\eta/2}\right)$. We shall explain the details in each case below, following exactly the arguments in~\cite[section 4]{harman1996on-the-distribu}. First, let us finish the proof of the lower bound if $\theta\le 2/7$. It follows that
\begin{equation*}
\begin{split}
\Phi(\tilde{\omega},\mathfrak{q}_0) &=( S_1-S_2-S_3+S_4+S_9-S_{10}+S_{11}+S_{13}+S_{14}-S_{15}+S_{16})+\\ & \quad + \tilde{S}_{13}-S_{13}+ O\left(\lambda N^{1-\eta/2}\right)\\
&\ge  (S_1-S_2-S_3+S_4+S_9-S_{10}+S_{11}+S_{13}+S_{14}-S_{15}+S_{16}) -S_{13}+\\ &  \quad + O\left(\lambda N^{1-\eta/2}\right)\\
 &= \Phi(\omega,\mathfrak{q}_0)-S_{13}+O\left(\lambda N^{1-\eta/2}\right),
\end{split}
\end{equation*}
where to obtain the last line, we have reversed all decompositions above (in particular those using Buchstab's identity).

We recall that
$$
S_{13}=\sum\limits_{(\overline{\nabla\nabla})} \Phi_{\mathfrak{pq}}(\omega,\mathfrak{q}).
$$
Since, by assumption in Theorem~\ref{lower}, $\omega=\lambda W$ with $W$ satisfying~\eqref{keycond}, \eqref{finitecond} and~\eqref{tailcond}, we may apply Theorem~\ref{genweights} to approximate  $S_{13}$ by
$$
S_{13}=\lambda\sum\limits_{(\overline{\nabla\nabla})} \mathcal{B}(u)\cdot \frac{y}{\log \mathcal{N}(\mathfrak{q})} \cdot \left(1+o(1)\right)
$$
as $N\rightarrow \infty$, where 
$$
y\coloneqq\frac{N}{\mathcal{N}(\mathfrak{pq})} \quad \mbox{and} \quad u\coloneqq\frac{\log y}{\log \mathcal{N}(\mathfrak{q})}.
$$
We also have
\begin{equation} \label{Phiomegaasymp}
\Phi_{\mathcal{O}}(\omega,\mathfrak{q}_0)=\lambda\cdot \frac{N}{\log N}\cdot (1+o(1))
\end{equation}
by the same Theorem~\eqref{genweights}. Hence,
\begin{equation*}
\Phi_{\mathcal{O}}(\tilde{\omega},\mathfrak{q}_0)=\lambda\left(\frac{N}{\log N}-\sum\limits_{(\overline{\nabla\nabla})} \mathcal{B}(u)\cdot \frac{y}{\log \mathcal{N}(\mathfrak{q})}\right) \cdot \left(1+o(1)\right).
\end{equation*}
To approximate the sum $\sum_{(\overline{\nabla\nabla})}$ on the right-hand side, we use partial summation together with Landau's prime ideal theorem (Lemma~\ref{LPIT}) in the same fashion as in~\cite[section 4]{harman1996on-the-distribu}. In the case when $1/4\le \theta\le 2/7$, this leads to
$$
\Phi_{\mathcal{O}}(\tilde{\omega},\mathfrak{q}_0)\ge \lambda\cdot \frac{N}{\log N}\cdot C(\theta)(1+o(1)), 
$$ 
where 
\begin{equation} \label{C1}
\begin{split}
C(\theta)\coloneqq\Bigg(1- &
\int\limits_{1-2\theta}^{1/2} \int\limits_{1-\theta-\alpha}^{(1-\alpha)/2}
\frac{d\beta d\alpha}{\alpha\beta(1-\alpha-\beta)}\\ 
- & \int\limits_{(1-\theta)/3}^{\theta} \int\limits_{(1-\theta-\alpha)/2}^{\alpha}
\mathcal{B}\left(\frac{1-\alpha-\beta}{\beta}\right)\frac{d\beta d\alpha}{\alpha\beta^2}\Bigg).
\end{split}
\end{equation}
Together with~\eqref{Phiomegaasymp}, this gives
$$
\frac{\mathcal{S}_{\mathcal{O}}(\tilde{\omega}, \sqrt{x})}{\mathcal{S}_{\mathcal{O}}(\omega, \sqrt{x})}=\frac{\Phi_{\mathcal{O}}(\tilde{\omega},\mathfrak{q}_0)}{\Phi_{\mathcal{O}}(\omega,\mathfrak{q}_0)}\ge C(\theta)+o(1).
$$ 
In~\cite[section 4]{harman1996on-the-distribu}, it was worked out that the above function $C(\theta)$ indeed satisfies
$$
C(\theta)=1+O\left((\theta-1/4)^2\right),
$$
as claimed in the theorem. 

Now we explain why the sums $\tilde{S}_i$ with $i=1,2,3,6,9,10,11,14,15,16$ are in a form which allows us to approximate them by the corresponding sums $S_i$ with $\omega$ in place of $\tilde{\omega}$ by applying~\eqref{typeII} or Theorem~\ref{asymp} in conjunction with~\eqref{typeI} and~\eqref{typeII}. If $R\le x^{1-\theta}/2$, then using~\eqref{typeI} and~\eqref{typeII} together with Theorem~\ref{asymp} with $M=x^{1-\theta}$, $\mu=\theta$ and $\kappa=1-3\theta$, we have 
\begin{equation} \label{consequence}
\sum\limits_{\mathcal{N}(\mathfrak{r})\sim R} c_{\mathfrak{r}} 
\Phi_{\mathfrak{r}}(\tilde{\omega},\mathfrak{q}_1)=\sum\limits_{\mathcal{N}(\mathfrak{r})\sim R}
c_{\mathfrak{r}}\Phi_{\mathfrak{r}}(\omega,\mathfrak{q}_1)+O\left(\lambda N^{1-\eta/2}\right)
\end{equation}
whenever $c_{\mathfrak{r}}\le 1$, where we recall that $\mathfrak{q}_1\coloneqq\mathfrak{Q}\left(x^{\kappa}\right)$.  Hence, we can immediately approximate $\tilde{S}_1$, $\tilde{S}_2$, $\tilde{S}_9$ and $\tilde{S}_{10}$ by the corresponding sums. Here we note that $\mathcal{N}(\mathfrak{pq}),\mathcal{N}(\mathfrak{pqr})\le\mathcal{N}(\mathfrak{p}\mathfrak{q}^2)< x^{1-\theta}$ in the sums $\tilde{S}_9$ and $\tilde{S}_{10}$. 

Since $x^{\theta}\le \mathcal{N}(\mathfrak{p})\le x^{1-2\theta}$ in $\tilde{S}_{3}$, this sum can be approximated directly using~\eqref{typeII} by disentangling variables using the M\"obius function and cosmetic surgery (Lemma~\ref{disentangle}) as in the proof of Theorem~\ref{asymp}. Similarly, the sum $\tilde{S}_6$ can be approximated using~\eqref{typeII}. Indeed, if $\beta\in [\theta,1-2\theta)$ or $ \alpha+\beta\in [\theta,1-2\theta)$, then $\mathcal{N}(\mathfrak{p})$ or $\mathcal{N}(\mathfrak{pq})$ lies in the correct range $[x^{\theta},x^{1-2\theta}]$. If $\beta\in [2\theta,1-\theta)$ or $\alpha+\beta\in [2\theta,1-\theta)$, then we reverse the roles of variables as follows: If $x^{2\theta}\le \mathcal{N}(\mathfrak{a})\le x^{1-\theta}$ in our Type II sum, then  $x^{\theta-\iota}\le \mathcal{N}(\mathfrak{b})\le x^{1-2\theta}$ unless $\mathcal{N}(\mathfrak{ab})\not\in [x^{1-\iota},x]$. The contribution of the latter $\mathfrak{ab}$'s can be bounded using~\eqref{newsmallandlarge}, and hence the rest can be handled using~\eqref{typeII} with the roles of $\mathfrak{a}$ and $\mathfrak{b}$ reversed. 

The sum $\tilde{S}_{11}$ counts certain products of five ideals $\mathfrak{pqrst}$. The norm of each of these ideals is at least $x^{1-3\theta}$. Hence, if $\theta\le 2/7$, then
$$
x^{\theta}\le x^{2-6\theta}\le \min\{\mathcal{N}(\mathfrak{rs}),\mathcal{N}(\mathfrak{st})\}< x^{2/5}<x^{1-2\theta}.
$$
Thus $\tilde{S}_{11}$ can be approximated using~\eqref{typeII} as well. 
In $\tilde{S}_{14}$ and $\tilde{S}_{15}$ we have $\mathcal{N}(\mathfrak{qt})< x^{1-\theta}$ and $\mathcal{N}(\mathfrak{qus})< x^{1-\theta}$, respectively, which again allows us to approximate these sums using~\eqref{consequence}. Finally, if $\theta\le 2/7$, then $\tilde{S}_{16}$ can be dealt with in a similar way as $\tilde{S}_{11}$. 
This completes the proof of the lower bound if $\theta\le 2/7$.

Since the proof of the lower bound for $\theta>2/7$ is parallel to that in~\cite[section 4]{harman1996on-the-distribu}, we cut the details short.  In this case, there are further losses coming from the sums $\tilde{S}_{11}$ and $\tilde{S}_{16}$ which make it necessary to subtract more integrals on the right-hand side of~\eqref{C1}. These integrals are of the form
$$
\int\limits_{\mathcal{D}} \mathcal{B}\left(\frac{1-\alpha-\beta-\gamma-\delta}{\delta}\right) \frac{d\alpha d\beta d\gamma d\delta}{\alpha\beta\gamma \delta^2}
$$
and
$$
\int\limits_{\mathcal{E}} \mathcal{B}\left(\frac{\alpha-\gamma}{\gamma}\right)
\mathcal{B}\left(\frac{1-\alpha-\beta-\delta}{\delta}\right) \frac{d\alpha d\beta d\gamma d\delta}{\beta \gamma^2 \delta^2},
$$
where $\mathcal{D}$ and $\mathcal{E}$ are rather complicated regions which can be found in~\cite[page 250]{harman1996on-the-distribu}. The numerical evaluation in~\cite[page 252]{harman1996on-the-distribu} then establishes the desired lower bound for $2/7<\theta\le 1/4$, and one has $C(7/22)>1/10$. 
\end{proof}
$ $\\
{\bf Remark 4:} In~\cite[section 4]{harman1996on-the-distribu}, Harman also established an upper bound. His proof started with the same equations corresponding to~\eqref{Buchtwice} and~\eqref{S4split} but then continued with a different decomposition. As $\theta \rightarrow 1/3$, the fundamental lemma of sieve theory is needed because in this case, Harman's asymptotic sieve has no Type II information to work with. We have refrained from working out a proof of the upper bound sieve in the number field setting as we need only a lower bound for our application.

\section{Checking conditions~\eqref{keycond}, \eqref{finitecond}, \eqref{tailcond}, \eqref{Xcond}, \eqref{newsmallandlarge}}
To apply Theorem~\ref{lower} to our sieve problem in real and imaginary quadratic number fields, we need to check that our conditions~\eqref{keycond}, \eqref{finitecond} and~\eqref{tailcond} for the relevant function $W$ hold in these settings. Moreover, we need to check that the conditions~\eqref{Xcond} and~\eqref{newsmallandlarge} for the relevant functions $\omega$ and $\tilde{\omega}$ in Theorem~\ref{lower} hold for a suitable $A>0$ and $\iota=2\varepsilon$. This will be carried out in the following.

\subsection{Real quadratic case}
We first check the conditions~\eqref{keycond}, \eqref{finitecond} and~\eqref{tailcond}. In~\cite{BM}, we considered the weight function  
\begin{equation} \label{Ws}
\Psi(\mathfrak{n})\coloneqq\sum\limits_{\substack{k\in \mathcal{O}\\ (k)=\mathfrak{n}}} f\left(\frac{\sigma_1(k)}{\sqrt{N}}\right)f\left(\frac{\sigma_2(k)}{\sqrt{N}}\right),
\end{equation}
where
\begin{equation} \label{fdef}
f(x)\coloneqq\left(\exp\left(-\pi x^2\right)-\exp\left(-2 \pi x^2\right)\right)^{\mathcal{C}}
\end{equation}
for some $\mathcal{C}\in \mathbb{N}$. Our function $W(\mathfrak{n})$ is precisely this function, scaled by some factor, i.e., 
\begin{equation} \label{Wrealdef}
W(\mathfrak{n})\coloneqq\frac{\Psi(\mathfrak{n})}{{\bf constant}} \mbox{ for all } \mathfrak{n}\in I
\end{equation}
for a suitable positive ${\bf constant}$ only depending on $\mathcal{C}$ and $\mathbb{K}$ which we will specify in \eqref{constantdefinition}. According to~\cite{BM}, the weight function $\omega$ is then defined as 
$$
\omega(\mathfrak{q})=\frac{\delta^2}{2\sqrt{d}}\cdot\Psi(\mathfrak{q})=\frac{\delta^2}{2\sqrt{d}}\cdot{\bf constant}\cdot W(\mathfrak{q}),
$$
where $\delta$ plays a similar role as in the imaginary-quadratic setting (here the approximation problem is two-dimensional, though) and $\mathbb{K}=\mathbb{Q}(\sqrt{d})$ with $d>1$ square-free. The function $\tilde{\omega}$ (which depends on $\delta$) is defined in \cite{BM} as follows: We write
\begin{align} \label{Omegadef}
    \Omega_{\Delta}(x):=\exp\left(-\pi \cdot \frac{x^2}{\Delta^2}\right)
\end{align}
and set 
\begin{equation*} 
    \tilde{\omega}(\mathfrak{q}):=\frac{N}{\mathcal{N}(\mathfrak{q})}\cdot \Psi(\mathfrak{q})\cdot  F(\mathfrak{q})=\frac{N}{\mathcal{N}(\mathfrak{q})}\cdot {\bf constant}\cdot W(\mathfrak{q})\cdot  F(\mathfrak{q})
\end{equation*}
with
\begin{equation*} 
F(\mathfrak{q}):=\sum\limits_{p\in\mathcal{O}}\Omega_{\delta/\sqrt{N}}\left(x_1-\frac{\sigma_1(p)}{\sigma_1(q)}\right)\Omega_{\delta/\sqrt{N}}\left(x_2-\frac{\sigma_2(p)}{\sigma_2(q)}\right),
\end{equation*}
where $q$ above is any generator of $\mathfrak{q}$, i.e. 
\begin{align*}
  \mathfrak{q}=(q).  
\end{align*}

For the sake of clarity, we first derive~\eqref{keycond} for $\mathfrak{r}=\mathcal{O}$. In this case, we need to consider the average
$$
\sum\limits_{\mathfrak{s}\in \mathbb{P}} \Psi(\mathfrak{s}).
$$
We relate this sum to 
$$
\sum\limits_{\mathfrak{n}\in \mathcal{I}} \Lambda(\mathfrak{n}) 
\Psi(\mathfrak{n}),
$$
where $\Lambda(\mathfrak{n})$ is the analog of the von Mangoldt function for ideals, defined by
$$
\Lambda(\mathfrak{n})=\begin{cases} \log \mathcal{N}(\mathfrak{p}) & \mbox{ if } \mathfrak{n}=\mathfrak{p}^k \mbox{ with } \mathfrak{p}\in \mathbb{P} \mbox{ and } k\in \mathbb{N},\\ 0 & \mbox{ otherwise.}
\end{cases}
$$
Separating the contribution of prime ideal powers $\mathfrak{p}^k$ with $k\ge 2$, we get
$$
\sum\limits_{\mathfrak{n}\in \mathcal{I}} \Lambda(\mathfrak{n}) 
\Psi(\mathfrak{n})=\sum\limits_{\mathfrak{s}\in \mathbb{P}} (\log \mathcal{N}(\mathfrak{s})) \Psi(\mathfrak{s}) + O\left(N^{1/2+\varepsilon}\right). 
$$ 
Using the bound $\Psi(\mathfrak{n})\ll \log N$ (see~\cite[(12)]{BM}), which is valid for all ideals $\mathfrak{n}\in I$, together with Landau's prime ideal theorem, we get
$$
\sum\limits_{\substack{\mathfrak{s}\in \mathbb{P}\\ \mathcal{N}(\mathfrak{s})\le N/\log^2 N}} (\log \mathcal{N}(\mathfrak{s})) \Psi(\mathfrak{s})  \ll 
\frac{N}{\log N}.
$$
We also have 
$$
\sum\limits_{\substack{\mathfrak{s}\in \mathbb{P}\\ \mathcal{N}(\mathfrak{s})> N\log^2 N}} (\log \mathcal{N}(\mathfrak{s})) \Psi(\mathfrak{s})  \ll  \frac{N}{\log N}
$$
since (see~\cite[(12)]{BM})
\begin{equation} \label{larges}
\Psi(\frak{n})\ll \exp\left(-\pi \mathcal{D} \mathcal{C} \cdot \frac{\mathcal{N}(\mathfrak{n})}{N}\right)\cdot \log \mathcal{N}(\mathfrak{n}) \mbox{ for all } \mathfrak{n}\in \mathcal{I}
\end{equation}
for some constant $\mathcal{D}>0$. 
It follows that 
$$
\sum\limits_{\mathfrak{n}\in \mathcal{I}} \Lambda(\mathfrak{n}) 
\Psi(\mathfrak{n}) = \sum\limits_{\substack{\mathfrak{s}\in \mathbb{P}\\ N/\log^2 N<\mathcal{N}(\mathfrak{s})\le N\log^2 N}} (\log \mathcal{N}(\mathfrak{s})) \Psi(\mathfrak{s})  + O\left( \frac{N}{\log N}\right).
$$
In the summation range 
$$
\frac{N}{\log^2 N}<\mathcal{N}(\mathfrak{s})\le N\log^2 N,
$$
we have
$$
\log \mathcal{N}(\mathfrak{s})=\log N + O(\log\log N).
$$
Moreover, by a similar process as above,
\begin{align*}
	\MoveEqLeft[4]
	\sum\limits_{\substack{\mathfrak{s}\in \mathbb{P}\\ N/\log^2 N<\mathcal{N}(\mathfrak{s})\le N\log^2 N}}  (\log N + O(\log\log N))\Psi(\mathfrak{s})\\ & =
 (\log N + O(\log\log N)) \sum\limits_{\mathfrak{s}\in \mathbb{P}} \Psi(\mathfrak{s}) + O\left(\frac{N\log \log N}{\log N}\right).
\end{align*}
We deduce that
$$
\sum\limits_{\mathfrak{n}\in \mathcal{I}} \Lambda(\mathfrak{n}) 
\Psi(\mathfrak{n}) =(\log N + O(\log\log N)) \sum\limits_{\mathfrak{s}\in \mathbb{P}} \Psi(\mathfrak{s}) + O\left(\frac{N\log\log N}{\log N}\right)
$$
and hence
\begin{equation} \label{primesumlambdasum}
\sum\limits_{\mathfrak{s}\in \mathbb{P}} \Psi(\mathfrak{s})=
\frac{1}{\log N + O(\log\log N)}\sum\limits_{\mathfrak{n}\in \mathcal{I}} \Lambda(\mathfrak{n}) 
\Psi(\mathfrak{n}) + O\left(\frac{N\log\log N}{\log^2 N}\right).
\end{equation}

Next we evaluate the sum on the right-hand side of~\eqref{primesumlambdasum}. Let $\epsilon$ be the fundamental unit.
Then, for any generator $k$ of $\mathfrak{n}$, we have 
\begin{equation} \label{Psiwrite}
\Psi(\mathfrak{n})=2\sum\limits_{n=-\infty}^{\infty} 
f\left(\frac{\epsilon^n|\sigma_1(k)|}{\sqrt{N}}\right)
f\left(\frac{\epsilon^{-n}|\sigma_2(k)|}{\sqrt{N}}\right),
\end{equation}
taking into account all positive and negative units and using the fact that $f$ is even. Since
$$
|\sigma_1(k)\sigma_2(k)|=\mathcal{N}(\mathfrak{n}),
$$
this can be re-written in the form
$$
\Psi(\mathfrak{n})=2\sum\limits_{n=-\infty}^{\infty} 
f\left(\frac{\epsilon^n}{\sqrt{N}} g_1\left(\frac{\log |\sigma_1(k)/\sigma_2(k)|}{2\log \epsilon}\right)\right)
f\left(\frac{\epsilon^{-n}}{\sqrt{N}} g_2\left(\frac{\log |\sigma_1(k)/\sigma_2(k)|}{2\log \epsilon}\right)\right),
$$
where
$$
g_1(\theta)\coloneqq\epsilon^{\theta} \cdot \sqrt{\mathcal{N}(\mathfrak{n})} \quad \mbox{and}
\quad g_2(\theta)\coloneqq\epsilon^{-\theta}\cdot \sqrt{\mathcal{N}(\mathfrak{n})}.
$$
Now we define
$
G: \mathbb{R}\times \mathbb{R}_{\ge 0} \rightarrow \mathbb{R}
$
by
$$
G(\theta,x)=\sum\limits_{n=-\infty}^{\infty} f\left(\epsilon^{n+\theta}\sqrt{x}\right)f\left(\epsilon^{-(n+\theta)}\sqrt{x}\right).
$$
Then it follows that 
$$
\Psi(\mathfrak{n})=2G\left(\frac{\log |\sigma_1(k)/\sigma_2(k)|}{2\log \epsilon},\frac{\mathcal{N}(\mathfrak{n})}{N}\right).
$$

Clearly, $G(\theta,x)$ is periodic in $\theta$ with period 1 and hence has a Fourier series development of the form 
$$
G(\theta,x)=\sum\limits_{m=-\infty}^{\infty} c_m(x)e(m\theta).
$$
It follows that
$$
\Psi(\mathfrak{n})=2\sum\limits_{m=-\infty}^{\infty} c_m\left(\frac{\mathcal{N}(\mathfrak{n})}{N}\right)e\left(m\cdot\frac{\log |\sigma_1(k)/\sigma_2(k)|}{2\log \epsilon}\right).
$$ 
The exponentials on the right-hand side are precisely the Hecke Gr\"o\ss encharaktere (see~\cite{Hecke1} and~\cite{Hecke2})
$$
\lambda^m(\mathfrak{n})\coloneqq e\left(m\cdot\frac{\log |\sigma_1(k)/\sigma_2(k)|}{2\log \epsilon}\right).
$$
We note that the right-hand side indeed only depends on the ideal $\mathfrak{n}$. Hence, in shorter form, we may write
$$
\Psi(\mathfrak{n})=2\sum\limits_{m=-\infty}^{\infty} c_m\left(\frac{\mathcal{N}(\mathfrak{n})}{N}\right)\lambda^m(\mathfrak{n}).
$$

The Fourier coefficient $c_m(x)$ equals
\begin{equation*}
\begin{split}
c_m(x) &= \int\limits_0^1 G(\theta,x)e(-m\theta) d\theta\\
&= \sum\limits_{n=-\infty}^{\infty} \int\limits_0^1  f\left(\epsilon^{n+\theta}\sqrt{x}\right)f\left(\epsilon^{-(n+\theta)}\sqrt{x}\right) e(-m\theta) d\theta \\
&=  \int\limits_{-\infty}^{\infty}   f\left(\epsilon^{y}\sqrt{x}\right)f\left(\epsilon^{-y}\sqrt{x}\right) e(-my)dy.
\end{split}
\end{equation*}
Therefore, we get
$$
\sum\limits_{\mathfrak{n}\in \mathcal{I}} \Lambda(\mathfrak{n}) \Psi(\mathfrak{n})=2\sum\limits_{m=-\infty}^{\infty} \int\limits_{-\infty}^{\infty} e(-my) \left(\sum\limits_{\mathfrak{n}\in \mathcal{I}} F_y\left(\frac{\mathcal{N}(\mathfrak{n})}{N}\right) \Lambda(\mathfrak{n}) \lambda^m(\mathfrak{n})\right)dy,
$$
where 
$$
F_y(x)\coloneqq f\left(\epsilon^{y}\sqrt{x}\right)f\left(\epsilon^{-y}\sqrt{x}\right).
$$
Now let $\phi_y(s)$ be the Mellin transform of $F_y(x)$, i.e.,
$$
\phi_y(s)=\int\limits_0^{\infty} x^{s-1}F_y(x)dx.
$$
Then, by Mellin inversion formula,
$$
F_y(x)=\frac{1}{2\pi i} \int\limits_{c-i\infty}^{c+i\infty} x^{-s}\phi_y(s)ds
$$
for any $c>1$.
Hence, we obtain
\begin{align*}
	\MoveEqLeft[1]
	\sum\limits_{\mathfrak{n}\in \mathcal{I}} \Lambda(\mathfrak{n}) \Psi(\mathfrak{n})\\ &= 2\sum\limits_{m=-\infty}^{\infty} \int\limits_{-\infty}^{\infty} e(-my) \left(\sum\limits_{\mathfrak{n}\in \mathcal{I}}  \left(\frac{1}{2\pi i} \int\limits_{c-i\infty}^{c+i\infty} \left(\frac{\mathcal{N}(\mathfrak{n})}{N}\right)^{-s}\phi_y(s)ds\right) \Lambda(\mathfrak{n}) \lambda^m(\mathfrak{n})\right)dy\\
&=  2\sum\limits_{m=-\infty}^{\infty} \frac{1}{2\pi i}\int\limits_{c-i\infty}^{c+i\infty} N^{s} \left(\sum\limits_{\mathfrak{n}\in \mathcal{I}}  \mathcal{N}(\mathfrak{n})^{-s}\Lambda(\mathfrak{n}) \lambda^m(\mathfrak{n})\right)\left(\int\limits_{-\infty}^{\infty} \phi_y(s)e(-my)dy\right)ds\\
&= - 2\sum\limits_{m=-\infty}^{\infty} \frac{1}{2\pi i}\int\limits_{c-i\infty}^{c+i\infty} N^{s}\cdot \frac{L'}{L}(s,\lambda^m)\left(\int\limits_{-\infty}^{\infty} \phi_y(s)e(-my)dy\right)ds,
\end{align*}
where
$$
L(s,\lambda^m)\coloneqq\sum\limits_{\mathfrak{n}\in \mathcal{I}}  \lambda^m(\mathfrak{n})\mathcal{N}(\mathfrak{n})^{-s}
$$
is the Hecke $L$-function associated to the Gr\"o\ss encharakter $\lambda^m$. 

Next, we want to find the Mellin transform of $F_y(x)$. Using the definition of $f$ in~\eqref{fdef} and multiplying out, we get
\begin{equation*}
\begin{split} 
F_y(x) &
\coloneqq \left(\exp\left(-\pi \epsilon^{2y} x\right)-\exp\left(-2 \pi \epsilon^{2y}x\right)\right)^{\mathcal{C}}\left(\exp\left(-\pi \epsilon^{-2y} x\right)-\exp\left(-2 \pi \epsilon^{-2y}x\right)\right)^{\mathcal{C}}\\
&= \left(\sum\limits_{a=0}^{\mathcal{C}} (-1)^a\binom{\mathcal{C}}{a} \exp\left(-\pi (2\mathcal{C}-a)\epsilon^{2y} x\right)\right)\left(\sum\limits_{b=0}^{\mathcal{C}} (-1)^b\binom{\mathcal{C}}{b} \exp\left(-\pi (2\mathcal{C}-b)\epsilon^{-2y} x\right)\right)\\
&= \sum\limits_{a=\mathcal{C}}^{2\mathcal{C}}  \sum\limits_{b=\mathcal{C}}^{2\mathcal{C}} (-1)^{a+b} \binom{\mathcal{C}}{2\mathcal{C}-a}\binom{\mathcal{C}}{2\mathcal{C}-b} \exp\left(-\pi \left(a\epsilon^{2y}+b\epsilon^{-2y}\right) x\right).
\end{split}
\end{equation*}
Hence, the Mellin transform is
\begin{equation*} 
\phi_y(s)= \pi^{-s} \Gamma(s) \sum\limits_{a=\mathcal{C}}^{2\mathcal{C}}  \sum\limits_{b=\mathcal{C}}^{2\mathcal{C}} (-1)^{a+b} \binom{\mathcal{C}}{2\mathcal{C}-a}\binom{\mathcal{C}}{2\mathcal{C}-b} \left(a\epsilon^{2y}+b\epsilon^{-2y}\right)^{-s}.
\end{equation*}
We deduce that
\begin{equation} \label{compex}
\begin{split}
\sum\limits_{\mathfrak{n}\in \mathcal{I}} \Lambda(\mathfrak{n}) \Psi(\mathfrak{n}) 
&= - 2\sum\limits_{a=\mathcal{C}}^{2\mathcal{C}}  \sum\limits_{b=\mathcal{C}}^{2\mathcal{C}} (-1)^{a+b} \binom{\mathcal{C}}{2\mathcal{C}-a}\binom{\mathcal{C}}{2\mathcal{C}-b} \sum\limits_{m=-\infty}^{\infty} \frac{1}{2\pi i}\times\\ & \int\limits_{c-i\infty}^{c+i\infty} N^{s}\pi^{-s} \Gamma(s)\cdot \frac{L'}{L}(s,\lambda^m)\left(\int\limits_{-\infty}^{\infty} \left(a\epsilon^{2y}+b\epsilon^{-2y}\right)^{-s}e(-my)dy\right)ds.
\end{split}
\end{equation}

The function $L(s,\lambda_0)$ equals the Dedekind zeta function and hence has a simple pole at $s=1$. If $m\not=0$, then $L(s,\lambda^m)$ is entire. It is known (see~\cite{Fogels},~\cite[chapter 5]{IwKo}) that $L(s,\lambda^m)$ has no zeros in the set
$$
\left\{ s\in \mathbb{C} \ :\ \Re s\ge 1-\frac{\varepsilon}{\log(2+|\Im s|)}\right\},
$$
and satisfies a bound of the form
\begin{equation} \label{logderiva}
\frac{L'}{L}(s,\lambda^m)\ll \log^2(2+|\Im s|)
\end{equation}
there if $\varepsilon>0$ is small enough.  
We suppose that $N\ge 2$, choose $c=1+1/\log N$ and then replace the contour $(c-i\infty,c+i\infty)$ of integration on the right-hand side of~\eqref{compex} by the union $\mathcal{U}$ of line segments $(c-i\infty,c-iT]$, $[c-iT,u-iT]$, $[u-iT,u+iT]$, $[u+iT,c+iT]$ and $[c+iT,c+i\infty)$, where 
$$
T\coloneqq\log N
$$ 
and  
$$
u=1-\frac{\varepsilon}{\log(2+T)}.
$$
Now using Cauchy's residue theorem, we get
\begin{equation*}
\sum\limits_{\mathfrak{n}\in \mathcal{I}} \Lambda(\mathfrak{n}) \Psi(\mathfrak{n}) 
= \mbox{\bf constant}\cdot N-E,
\end{equation*}
where 
\begin{equation} \label{constantdefinition}
{\bf constant}\coloneqq \frac{2}{\pi}\cdot \sum\limits_{a=\mathcal{C}}^{2\mathcal{C}}  \sum\limits_{b=\mathcal{C}}^{2\mathcal{C}} (-1)^{a+b} \binom{\mathcal{C}}{2\mathcal{C}-a}\binom{\mathcal{C}}{2\mathcal{C}-b} \int\limits_{-\infty}^{\infty} \left(a\epsilon^{2y}+b\epsilon^{-2y}\right)^{-1}dy
\end{equation}
and 
\begin{equation*}
\begin{split}
E &\coloneqq 2\sum\limits_{a=\mathcal{C}}^{2\mathcal{C}}  \sum\limits_{b=\mathcal{C}}^{2\mathcal{C}} (-1)^{a+b} \binom{\mathcal{C}}{2\mathcal{C}-a}\binom{\mathcal{C}}{2\mathcal{C}-b}  \sum\limits_{m\not=0} \frac{1}{2\pi i} \times\\ & \quad\times
\int\limits_{\mathcal{U}} N^{s}\pi^{-s} \Gamma(s)\cdot \frac{L'}{L}(s,\lambda^m)\left(\int\limits_{-\infty}^{\infty} \left(a\epsilon^{2y}+b\epsilon^{-2y}\right)^{-s}e(-my)dy\right)ds.
\end{split}
\end{equation*}
For an estimation of the error term $E$, we use Stirling's formula to bound $\Gamma(s)$ and~\eqref{logderiva} to bound $L'/L(s)$. It remains to bound the Fourier transform
$$
\int\limits_{-\infty}^{\infty} \left(a\epsilon^{2y}+b\epsilon^{-2y}\right)^{-s}e(-my)dy.
$$

Let $s=\sigma + it$. We write the above integral in the form
$$
\int\limits_{-\infty}^{\infty} \left(a\epsilon^{2y}+b\epsilon^{-2y}\right)^{-\sigma }e\left(f_{a,b,m}(y)\right)dy
$$
with 
$$
f_{a,b,m}(y)\coloneqq-\frac{t}{2\pi} \cdot \log\left(a\epsilon^{2y}+b\epsilon^{-2y}\right)-my.
$$
To get an idea of the behavior of this exponential integral, we first calculate that the stationary phase points $y_0$ with $f_{a,b,m}'(y_0)=0$ satisfy
$$
\epsilon^{4y_0}= \frac{b}{a}\cdot \frac{t \log \epsilon-\pi m}{t\log \epsilon + \pi m}.
$$
For them to exist, the right-hand side needs to be positive which is the case if and only if 
$$
|t|>\frac{\pi}{\log \epsilon}\cdot |m|. 
$$
If 
$$
|t|\le \frac{\pi}{2\log \epsilon}\cdot |m| \quad \mbox{and} \quad \sigma\ge
\frac{1}{2},
$$
then, in a standard way, repeated integration by part gives
$$
\int\limits_{-\infty}^{\infty} \left(a\epsilon^{2y}+b\epsilon^{-2y}\right)^{-\sigma }e\left(f_{a,b,m}(y)\right)dy = O_{a,b,B}\left(|m|^{-B}\right),
$$
where $B$ is any positive real. If 
$$
|t|> \frac{\pi}{2\log \epsilon}\cdot |m| \quad \mbox{and} \quad \sigma\ge
\frac{1}{2},
$$
then we shall use just the trivial bound
$$
\int\limits_{-\infty}^{\infty} \left(a\epsilon^{2y}+b\epsilon^{-2y}\right)^{-\sigma }e\left(f_{a,b,m}(y)\right)dy = O_{a,b}(1).
$$

Combining the above bounds, Stirling's formula and~\eqref{logderiva}, we find that
$$
E\ll N^{1-\varepsilon/\log\log N}
$$
for a suitable $\varepsilon>0$. Hence, we have
\begin{equation} \label{realquadcase}
\sum\limits_{\mathfrak{n}\in \mathcal{I}} \Lambda(\mathfrak{n}) \Psi(\mathfrak{n}) 
= \mbox{\bf constant}\cdot N+O\left(N^{1-\varepsilon/\log\log N}\right).
\end{equation}
Moreover, lower bounds for $\Psi(\mathfrak{n})$ in~\cite{BM} together with Landau's prime ideal theorem and the above asymptotic estimate~\eqref{realquadcase} show that ${\bf constant}>0$. Combining~\eqref{primesumlambdasum} and~\eqref{realquadcase}, we obtain 
\begin{equation} \label{realquadcase1}
\sum\limits_{\mathfrak{s}\in \mathbb{P}} W(\mathfrak{s}) 
= \frac{N}{\log N}+O\left(\frac{N\log \log N}{\log^2 N}\right)
\end{equation}
if $W(\mathfrak{n})$ is defined as in~\eqref{Wrealdef},
establishing~\eqref{keycond} for $\mathfrak{r}=\mathcal{O}$. 

It is easy to modify the calculations above to derive~\eqref{keycond} for general $\mathfrak{r}\in \mathcal{O}$ satisfying $\mathcal{N}(\mathfrak{r})\le N/2$: Choose some generator $l$ of $\mathfrak{r}$ which satisfies
$$
|\sigma_1(l)|\asymp \sqrt{\mathcal{N}(\mathfrak{r})} \asymp
|\sigma_2(l)|
$$
and write 
\begin{equation*} \label{Ws*}
W(\mathfrak{rs})\coloneqq\sum\limits_{\substack{k\in \mathcal{O}\\ (k)=\mathfrak{s}}} f\left(c_1\cdot \frac{\sigma_1(k)}{\sqrt{N/\mathcal{N}(\mathfrak{r})}}\right)f\left(c_2\cdot \frac{\sigma_2(k)}{\sqrt{N/\mathcal{N}(\mathfrak{r})}}\right),
\end{equation*}
where 
$$
c_i\coloneqq\frac{\sigma_i(l)}{\sqrt{\mathcal{N}(\mathfrak{r})}}
\quad \mbox{for } i=1,2.
$$
Now the above calculations for the case $\mathfrak{r}=\mathcal{O}$ go through 
in this general case as well, where $\sqrt{N}$ is replaced by $\sqrt{N/\mathcal{N}(\mathfrak{r})}$ and the additional constants $c_1$ and $c_2$ above are taken into consideration. This establishes~\eqref{keycond}.  

The bound~\eqref{finitecond} is a consequence of 
\begin{equation*} 
W(\mathfrak{n})\ll_{\varepsilon} 1 \mbox{ if } \mathcal{N}(\mathfrak{n})\ge N^{\varepsilon/2},
\end{equation*}
and the bound~\eqref{tailcond} follows from~\eqref{larges}. Hence, we have established the required conditions~\eqref{keycond}, \eqref{finitecond} and~\eqref{tailcond} on $W$. 

Now we turn to the conditions~\eqref{Xcond} and~\eqref{newsmallandlarge} for the weight functions $w=\omega,\tilde{\omega}$, which we shall establish for $A=2$, $\iota=2\varepsilon$ and $\eta>0$ small enough, provided our constant $\mathcal{C}$ in~\eqref{fdef} is large enough.  It was proved in~\cite[(14)]{BM} 
that 
$$
w(\mathfrak{n})\ll \exp\left(-\mathcal{E}\cdot \frac{\mathcal{N}(\mathfrak{n})}{N}\right)\cdot \log N
$$
for a suitable constant $\mathcal{E}>0$ depending on $\mathcal{D}$. 
We recall that $N\sim x^{1-\varepsilon}$, $\lambda \asymp \delta^2$ and $d_5(\mathfrak{n})\ll \mathcal{N}(\mathfrak{n})^{\varepsilon}$. Moreover, we may assume $\delta\ge x^{-7/44}$ in order to establish Theorem~\ref{realcase}. Hence, to prove both~\eqref{Xcond} and~\eqref{newsmallandlarge} with the said choices of $A$, $\iota$ and $\eta$, it suffices to show that
$$
w(\mathfrak{n})\ll 
\mathcal{N}(\mathfrak{n})^{-1} \mbox{ if } \mathcal{N}(\mathfrak{n})< x^{1-2\varepsilon}.
$$ 
Using~\cite[(6),(7),(8),(11)]{BM}, we have the rough bound
$$
w(\mathfrak{n}) \ll N\mathcal{N}(\mathfrak{n})^{-1} \Psi(\mathfrak{n})
$$
in this range. Hence, it suffices to show that
\begin{equation} \label{neededWbound}
W(\mathfrak{n})\ll 
x^{-1} \mbox{ if } \mathcal{N}(\mathfrak{n})< x^{1-2\varepsilon},
\end{equation}
which we shall establish in the following. 

We start with~\eqref{Psiwrite} and note that the generator $k$ of $\mathfrak{n}$ can be chosen in such a way that
$$
|\sigma_1(k)|\asymp |\sigma_2(k)| \asymp \mathcal{N}(\mathfrak{n})^{1/2}.
$$
We further note that (see~\cite[(4)]{BM})
$$
|f(x)|\ll \min\left\{1, |x|^{2\mathcal{C}}\right\}.
$$
Hence, we have
\begin{equation*} 
\begin{split}
\Psi(\mathfrak{n})\ll & \sum\limits_{n=-\infty}^{\infty} 
\min\left\{1, \left|\frac{\mathcal{N}(\mathfrak{n})}{N}\right|^{\mathcal{C}}\cdot \epsilon^{2n\mathcal{C}}\right\}\cdot
\min\left\{1, \left|\frac{\mathcal{N}(\mathfrak{n})}{N}\right|^{\mathcal{C}}\cdot \epsilon^{-2n\mathcal{C}}\right\}\\
\ll & \sum\limits_{n=0}^{\infty} \left|\frac{\mathcal{N}(\mathfrak{n})}{N}\right|^{\mathcal{C}}\cdot \epsilon^{-2n\mathcal{C}}
\ll  x^{-\mathcal{C}\varepsilon}
\end{split}
\end{equation*}
if $\mathcal{N}(\mathfrak{n})\le x^{1-2\varepsilon}$.  Now we choose $\mathcal{C}$ larger than $1/\varepsilon$. Then~\eqref{neededWbound} follows, and the proof is complete. In~\cite{BM}, we chose $\mathcal{C}=\lceil 100/\varepsilon \rceil$, which consists with what we need here. 

\subsection{Imaginary quadratic case}
We first establish~\eqref{keycond}, \eqref{finitecond}, \eqref{tailcond} and~\eqref{Xcond} for the original functions $W$, $\omega$ and $\tilde{\omega}$ used in~\cite{BT}. In this paper, we worked in the ring of integers $\mathcal{O}$ rather than the set of ideals $\mathcal{I}$ throughout and hence defined our weight functions on $\mathcal{O}$ instead of $\mathcal{I}$. Since the unit group of $K=\mathbb{Q}(\sqrt{-d})$ is finite and the class number is supposed to be 1, there is no essential difference between the $\mathcal{O}$- and the $\mathcal{I}$-setting as far as the sieve part is concerned. It therefore suffices to check the said  conditions for the corresponding weight function $W$ on the set of ideals $I$, possibly scaled by a suitable factor to ensure that we get exactly the asymptotic in~\eqref{keycond}. This weight function is simply (cf.~\cite{BT})
$$
W(\mathfrak{n})=\pi \cdot \exp\left(-\pi \cdot \frac{\mathcal{N}(\mathfrak{n})}{N}\right), 
$$   
with scaling factor $\pi$. The function $\omega$ is then $\omega(\mathfrak{a})=\lambda W(\mathfrak{a})$ with $\lambda=4\delta^2$ for some parameter $\delta$ which specifies the weight function $\tilde{\omega}$. For the precise definition of $\tilde{\omega}$ in the setting of $\mathcal{O}$, we refer to~\cite{BT}. Here we just mention that it detects elements $n$ of $\mathcal{O}$ such that the distance of $n\alpha$ to the nearest element $a$ of $\mathcal{O}$ is not much larger than $\delta$.  

An application of partial summation, Lemma~\ref{LPIT} and integration by parts gives
\begin{equation*}
\begin{split}
\sum\limits_{\mathfrak{s}\in \mathbb{P}} W(\mathfrak{rs})&=
-\pi \int\limits_{2}^{\infty} \frac{d}{dt} \exp\left(-\pi\cdot \frac{\mathcal{N}(\mathfrak{r})t}{N}\right) \Big(\sum\limits_{\substack{\mathfrak{s}\in \mathbb{P}\\
\mathcal{N}(\mathfrak{s})\le t}} 1\Big)dt \\
&= -\pi \int\limits_{2}^{\infty} \frac{d}{dt}\exp\left(-\pi\cdot \frac{\mathcal{N}(\mathfrak{r})t}{N}\right) \left(\frac{t}{\log t}+O\left(\frac{t}{\log^2 t}\right)\right)dt\\
&= \pi \int\limits_{2}^{\infty} \exp\left(-\pi\cdot \frac{\mathcal{N}(\mathfrak{r})t}{N}\right) \frac{d}{dt}\left(\frac{t}{\log t}\right) dt+O\left(\frac{N/N(r)}{\log^2 (N/N(r))}\right)\\
&= \frac{N/\mathcal{N}(r)}{\log (N/\mathcal{N}(r))}+O\left(\frac{N/\mathcal{N}(r)}{\log^2 (N/\mathcal{N}(r))}\right)
\end{split}
\end{equation*}
if $\mathcal{N}(\mathfrak{r})\le N/2$, establishing~\eqref{keycond}. The condition~\eqref{finitecond} follows immediately from $W(\mathfrak{a})\le 1$ and Lemma~\ref{LPIT}, without dependence on $\varepsilon$. Moreover, using partial summation, we easily establish the stronger bound 
\begin{equation*} 
\sum\limits_{\substack{\mathfrak{a}\in \mathcal{O}\\ \mathcal{N}(\mathfrak{ar})>\tilde{N}}} W(\mathfrak{ar})=O_{\xi}\left(\frac{N}{\mathcal{N}(\mathfrak{r})}\cdot \exp\left(-(\log N)^{\xi/2}\right)\right)
\end{equation*}
in place of~\eqref{tailcond} provided that $\tilde{N}=N(\log N)^{\xi}$ and $\mathcal{N}(\mathfrak{r})\le N/2$. Hence, we have established the required conditions
\eqref{keycond}, \eqref{finitecond} and~\eqref{tailcond} on $W$. Clearly, \eqref{Xcond} holds as well with $A=2$. 

However, with the above choices of $W$, $\omega$ and $\tilde{\omega}$, the condition~\eqref{newsmallandlarge} will not hold since we need $W(\mathfrak{n})$ to be small if $\mathcal{N}(\mathfrak{n})$ is small compared to $x$. The obvious solution to this problem is to modify $W$ (and correspondingly, $\tilde{\omega}$) in a similar way as in the real quadratic case, namely to choose
\begin{equation} \label{redef}
W(\mathfrak{n})=\frac{1}{\mbox{\bf constant}} \left(\exp\left(-\pi \cdot \frac{\mathcal{N}(\mathfrak{n})}{N}\right)-\exp\left(-2\pi \cdot \frac{\mathcal{N}(\mathfrak{n})}{N}\right)\right)^{\mathcal{C}}
\end{equation}
for a suitable {\bf constant} and $\mathcal{C}\in \mathbb{N}$. This will allow us, similarly as in the real quadratic case, to establish~\eqref{newsmallandlarge} for $\iota=2\varepsilon$ and small enough $\eta$. (In fact, the estimations are easier here.) To establish 
\eqref{keycond}, \eqref{finitecond} and~\eqref{tailcond} for this modified function, we simply open up the $\mathcal{C}$-th power in~\eqref{redef} by applying the binomial formula and use similar estimations as above for the relevant sums of the resulting terms. The bound~\eqref{Xcond} remains valid with $A=2$, of course.  
  
\section{Application to restricted Diophantine approximation} \label{quadproofs}
To prove Theorems~\ref{imag} and~\ref{realcase}, it remains to modify the final arguments in~\cite{BT} and~\cite{BM}. Now we use Harman's lower bound sieve in place of his asymptotic sieve, which we are allowed to do since we have verified the conditions~\eqref{keycond}, \eqref{finitecond}, \eqref{tailcond}, \eqref{Xcond} and~\eqref{newsmallandlarge} in the last section. The underlying Type I and Type II sum estimates remain the same as in~\cite{BT} and~\cite{BM}.
  
\subsection{Proof of Theorem~\ref{imag}} \label{imagcase}
	We use the setup in~\cite{BT}. Since our field $\mathbb{K}$ has class number one and finite unit group, we are again free to switch between elements and ideals 
in $\mathcal{O}$ since it is a PID. 
	Applying~\cite[Propositions 6.6 and 6.7]{BT} with $M=x^{1-\theta}$, $\mu=\theta-2\varepsilon$, $\mu+\kappa=1-2\theta$ where $1/4+2\varepsilon\le \theta\le 1/3$, we get that the sum of the Type I and Type II sums is bounded by
	\begin{equation*}
		\ll_{C,\varepsilon,\omega} \delta^{2} N
		\cdot x^{9\varepsilon} \bigl(
		\abs{q}^{2}x^{-1}
	        + \delta^{-2} \abs{q}^{-1} + \delta^{-1}\abs{q}x^{-1/2}+ \delta^{-2} x^{-\theta} 
		+ \delta^{-1} x^{-\theta/2} \bigr).
	\end{equation*}
	Upon taking $x = \abs{q}^{1/\theta}$ (and hence $|q|=x^{\theta}$) and recalling that $N=x^{1-\varepsilon}$ in~\cite{BT}, we find that the above is dominated by
	\[
		\ll_{C,\varepsilon,\omega} \delta^{2} N^{1-\varepsilon}
	\]
	provided that $\tfrac{1}{2}\geq\delta \geq x^{-\theta/2+12\varepsilon}$. If $\theta=7/22$, we then obtain the lower bound
$$
S_{\mathcal{O}}(\tilde{\omega},\sqrt{x}) \gg \delta^2 \frac{N}{\log N}
$$
from~\cite[Lemma 6.1.]{BT} and Theorem~\ref{lower} with 
 $$
 \lambda=\delta^2\cdot \mbox{\bf constant}'
 $$ 
 for a suitable constant $\mbox{\bf constant}'>0$ coming from the modification of the weight function $W$ in~\eqref{redef}. Now Theorem~\ref{imag} follows after cutting the series defining 
$S(\tilde{\omega},\sqrt{x})$ at $x$ as in~\cite[Lemma 6.2]{BT}.

\subsection{Proof of Theorem~\ref{realcase}} 
Here we use the setup in~\cite{BM}.  
	Again, we take $M=x^{1-\theta}$, $\mu=\theta-2\varepsilon$, $\mu+\kappa=1-2\theta$ where $1/4+2\varepsilon\le \theta\le 1/3$. We first note that the term $x^{3/4}$ in~\cite[equation (87)]{BM} can be replaced by $x^{1-\theta+2\varepsilon}$ because $K$ is now in the range $x^{\theta-2\varepsilon}\ll K\ll x^{1-2\theta}\le x^{1/2}$. Then proceeding along the lines in~\cite[section 11]{BM}, we arrive at   
\begin{equation} \label{diffes} 
\begin{aligned}
	\MoveEqLeft
	x^{-15\varepsilon}\cdot (\mbox{sum of Type I and Type II sums})\\
& \ll x^{1-\theta}W^{4\eta}+\delta^{-2}x^{2-\theta}W^{4\eta-2}+xW^{2\eta-1}+\delta^2W^{2\eta+1}
+\\ & \quad +  \left(\delta+x^{1/2}W^{-1}\right)\left(x^{1-\theta/2}W^{2\eta}+x^{1/2}W^{1/2+\eta}\right)+xW^{\eta-1/2}+x^{1-\theta}
\end{aligned}
\end{equation}
in place of~\cite[equation (126)]{BM}.  As in~\cite[section 11]{BM}, we choose $x$ depending on $W$ in such a way $\delta=x^{1/2}W^{-1}$, i.e.  
$$
x\coloneqq(\delta W)^2
$$
and hence
$$
W\asymp x^{1/2}\delta^{-1}.
$$
Then~\eqref{diffes} turns into
\begin{align*}\MoveEqLeft
	x^{-15\varepsilon}\cdot (\mbox{sum of Type I and Type II sums})\\
&\ll x^{1-\theta+2\eta}\delta^{-4\eta}+
x^{1-\theta/2+\eta}\delta^{1-2\eta}+x^{3/4+\eta/2}\delta^{1/2-\eta}
+x^{1-\theta}.
\end{align*}
Recalling that $N\coloneqq\left\lceil x^{1-\varepsilon}\right\rceil$ in~\cite{BM}, the estimate 
$$
\mbox{sum of Type I and Type II sums}\ll \delta^2N^{1-\varepsilon}
$$
follows if $\varepsilon\le 1/16$ and 
\begin{equation} \label{deltacondi}
\delta\ge N^{-\nu+17\varepsilon}
\end{equation}
with 
$$
\nu\coloneqq\min\left\{\frac{\theta/2-\eta}{1+2\eta},\frac{1/4-\eta/2}{3/2+\eta}\right\}.
$$
If $\theta\le 1/3$, then the first term in the minimum is less than the second term. Fixing $\theta\coloneqq7/22$, it follows that 
\begin{equation} \label{nu}
\nu=\frac{7/44-\eta}{1+2\eta}.
\end{equation}
As in subsection~\ref{imagcase}, we now obtain the desired lower bound
$$
S_{\mathcal{O}}(\tilde{\omega},\sqrt{x}) \gg \delta^2 \frac{N}{\log N}
$$
from Theorem~\ref{lower} with
$$
\lambda=\frac{\delta^2}{2\sqrt{d}}\cdot \mbox{\bf constant},
$$
provided $\delta$ satisfies~\eqref{deltacondi}. Using~\cite[equation (26)]{BM}, we get a sharpened version of~\cite[Theorem 6]{BM} with $\nu$ as in~\eqref{nu} above which leads to Theorem~\ref{realcase} using the same ``unsmoothing'' procedure as in~\cite[section 12]{BM}.

\end{document}